\documentclass[]{article}
\usepackage{amsmath, amsthm, amssymb,amscd, verbatim, enumerate,mathrsfs}

\usepackage{geometry}
\usepackage{amssymb}
\usepackage{amsmath}
\usepackage{amsthm}
\usepackage{amsfonts}
\usepackage{tikz}
\usetikzlibrary{matrix}
\usepackage[latin1]{inputenc}

\newcommand\restr[2]{{
  \left.\kern-\nulldelimiterspace 
  #1 
  \vphantom{\big|} 
  \right|_{#2} 
  }}
  
\newcommand{\binomi}[2]{\genfrac{(}{)}{0pt}{}{#1}{#2}}  
  
\newtheorem{theorem}{Theorem}[section]
\newtheorem{lemma}[theorem]{Lemma}
\newtheorem{proposition}[theorem]{Proposition}
\newtheorem{corollary}[theorem]{Corollary}

\theoremstyle{definition}
\newtheorem{definition}[theorem]{Definition}
\newtheorem{remark}[theorem]{Remark}

\def\X{\mathbf{X}}

\begin{document}


\author{Rosa Orellana\\
\\
Dartmouth College \\
Mathematics Department \\
6188 Kemeny Hall \\
Hanover, NH 03755, USA.\\
Rosa.C.Orellana@dartmouth.edu\\
\and
Geoffrey Scott\\
\\
University of Michigan\\
Mathematics department\\
4080 East Hall, 530 Church Street\\
Ann Arbor, MI 48109, USA.\\
gsscott@umich.edu\\
}

\title{Graphs with Equal Chromatic Symmetric Functions}

\maketitle

\begin{abstract}

Stanley \cite{bST} introduced the chromatic symmetric function $\X_G$ associated to a simple graph $G$ as a generalization of the chromatic polynomial of $G$.    In this paper we present a novel technique to write $X_G$ as a linear combination of chromatic symmetric functions of smaller graphs.   We use this technique to give a sufficient condition for two graphs to have the same chromatic symmetric function.  We then construct an infinite family of pairs of unicyclic graphs with the same chromatic symmetric function, answering the question posed by Martin, Morin, and Wagner \cite{bMMA} of whether such a pair exists.  Finally, we approach the problem of whether it is possible to determine a tree from its chromatic symmetric function. Working towards an answer to this question, we give a classification theorem for single-centroid trees in terms of data closely related to its chromatic symmetric function.
\medskip

\end{abstract}


\section*{Introduction}

In 1995, Stanley \cite{bST} introduced a symmetric function ${\bf X}_G = {\bf X}_G(x_1, x_2, \ldots)$ associated with any simple graph $G$ (see Section 1 for a precise definition) called the \emph{chromatic symmetric function} of $G$.  ${\bf X}_G$  has the property that when we specialize the variables to $x_1=\cdots = x_k=1$ and $x_i=0$ for all $i>k$ then ${\bf X}_G$ gives the number of ways to properly color the vertices of $G$ with $k$ colors.   Hence ${\bf X}_G(1,1,\ldots, 1, 0, \ldots)= \chi_G(k)$, where $\chi_G$ is the chromatic polynomial of $G$. 

One of the first questions posed by Stanley was whether ${\bf X}_G$ determines $G$.  As expected this is not the case, and Stanley provides the example of the kite and the bowtie as nonisomorphic graphs with the same ${\bf X}_G$ \cite[Fig. 1]{bST}.   Although two nonisomorphic graphs may share the same chromatic symmetric function, Stanley conjectured that two nonisomorphic trees must have distinct chromatic symmetric functions.  This conjecture is claimed to be true for trees with fewer than 23 vertices. This claim is found in the introduction of \cite{bMMA} and they cite Li-Yang Tang; however, the website containing this information is no longer available. Evidence that Stanley's conjecture is true has been found by Morin \cite{bMo} and Fougere \cite{bFo} who showed that some families of trees are determined by the chromatic symmetric function.   Martin, Morin and Wagner \cite{bMMA}  showed that the degree sequence and path sequence of a tree, $T$,  can be obtained  from  ${\bf X}_T$. They also show that some families of trees, called caterpillars and spiders, can be determined from their chromatic symmetric function. 

A fundamental property of the chromatic polynomial is the \emph{deletion-contraction} property, which allows us to write $\chi_G(k)$ as a linear combination of the chromatic polynomial of graphs with fewer edges. This property is the basis for inductive proofs of many other properties of the chromatic polynomial.  Unfortunately ${\bf X}_G$ does not satisfy a deletion-contraction law which makes it difficult to apply the useful technique of induction.   Gebhard and Sagan \cite{bGS} introduced a non-commutative version of $\X_G$ that satisfies the deletion-contraction property and is a complete invariant of graphs. 
One of our results is a novel technique to decompose  ${\bf X}_G$ as a linear combination of chromatic symmetric functions of other graphs. And in the case when $G$ has a triangle we can write $X_G$ as a linear combination of chromatic symmetric functions of graphs  with fewer edges than $G$. 

There are many properties of $G$ that can be recovered from ${\bf X}_G$. These include the number of vertices, the number of connected components, the number of matchings, and the girth.  We have found that the number of triangles in $G$ can also be recovered from ${\bf X}_G$.   In the case that the graph is a tree, $T$, a lot more can be recovered from ${\bf X}_T$; for example,  the degree sequence can be recovered from ${\bf X}_T$ \cite{bMMA}.  This is no longer true for general graphs, we provide an example of a pair of non-isomorphic graphs with the same ${\bf X}_G$ but different degree sequences.  Although, the degree sequence can no longer be recovered from ${\bf X}_G$ for arbitrary $G$, we show that the sum of the squares of the degrees can be recovered from ${\bf X}_G$.  This is a generalization of a result in \cite{bFo} that shows the analogous result for trees. 

In \cite{bMMA} the authors showed that ${\bf X}_G$ is an complete invariant for two special families of unicyclic graphs and ask whether there exists a pair of unicyclic graphs with the same ${\bf X}_G$.   We answer this question in the affirmative by giving a pair of unicyclic graphs with the same chromatic symmetric function.   In fact, our Theorem \ref{P1} gives a sufficient condition for two graphs to have the same chromatic symmetric function.  We apply this theorem to construct infinitely many pairs of unicyclic graphs with the same ${\bf X}_G$.  The same technique can also be used to construct pairs of general graphs with the same ${\bf X}_G$.  We have also studied trees and we give a classification theorem for trees with one centroid.  This classification arose from our study of the chromatic symmetric function of a tree when written in the power-sum symmetric basis.   

Our paper is organized as follows. In Section 1 we review background information, set up notation, and define the chromatic symmetric function.   In Section 2 we look at properties of $G$ that are determined by $\X_G$ for general graphs. In particular, we show that the sum of the squares of the degrees as well as the number of triangles in a graph can be recovered from the chromatic symmetric function.   In Section 3 we show how the chromatic symmetric function of a graph can be written as a linear combination of other chromatic symmetric functions.   In Section 4 we focus our attention on unicyclic graphs.  We also prove a sufficient condition for two graphs to have the same chromatic symmetric function and show how to construct pairs of graphs with the same ${\bf X}_G$.  In our last section, Section 5,  we prove a classification theorem for trees with a single centroid that is closely related to the coefficients of the chromatic symmetric function when written in the power-sum symmetric basis.


\section{Preliminaries}
We assume that the reader is familiar with the basic facts about graphs found in any introductory graph theory book (see e.g., \cite{bBo, bDI,bHA}).  In this section we establish notation that will be used throughout the paper.   A graph $G$ is an ordered pair $(V,E)$, where $V=V(G)$ is the vertex set and $E=E(G)$ is the edge set.  All our graphs are \emph{simple}, i.e., we do not allow loops or multiple edges.   The number of vertices $\#V(G)$ is called the \emph{order} of the graph. 
We will write $uv$ for the edge joining the vertices $u, v \in V(G)$ if such an edge exists. We say that $u$ and $v$ are {\em endpoints} of $uv$, that $uv$ is {\em incident} to $u$ and $v$, and that $u$ is {\em adjacent} to $v$. If two edges have no endpoints in common, they are {\em disjoint}.   The \emph{degree} $d(v)$ of a vertex $v$ is the number of edges incident to $v$. The \emph{degree sequence} of a graph $G$ is the sequence $(d(v))_{v \in V(G)}$.  An \emph{isolated} vertex is a  vertex of degree 0. A \emph{leaf} is a vertex of degree 1.   The \emph{girth} of a graph is the number of distinct vertices in a shortest cycle in the graph.  An acyclic graph has infinite girth.

A \emph{subgraph} $G^\prime \subseteq G$ of a graph $G = (V(G), E(G))$ is a graph $G^\prime = (V^\prime(G), E^\prime(G))$ such that $V^\prime(G) \subseteq V(G)$ and $E^\prime(G) \subseteq E(G)$. A subgraph is said to be \emph{induced} by the vertex set $V^\prime(G)$ if every edge in $E(G)$ having endpoints in $V^\prime(G)$ is also in $E^\prime(G)$.   A subgraph $H$ is a \emph{spanning subgraph} of $G$ if it has the same vertex set as $G$.
A subgraph is said to be a \emph{matching} of size $k$ if it consists of $k$ disjoint edges on $2k$ vertices. 

In this paper we are interested in certain classes of simple graphs. A graph is called \emph{unicyclic} if it contains exactly one cycle, a \emph{forest} if it contains no cycles, and a \emph{tree} if it is a connected forest.  Notice that a connected unicyclic graph with $n$ vertices has $n$ edges.

In the following proposition we summarize some well-known facts about trees.   The reader may refer to \cite{bBo,bDI,bHA}  or any other introductory graph theory textbook for proofs of these facts.

\begin{proposition}[\cite{bBo}, pp. 99-100]
\begin{enumerate}
\item[(1)] In a tree, any two vertices are connected by exactly one path.
\item[(2)] Every tree on $n$ vertices has $n-1$ edges. In general, a forest on $n$ vertices with $c$ connected components has $n-c$ edges.
\item[(3)] Every nontrivial tree has at least two leaves. In general, if a forest contains $c$ connected nontrivial components, then it contains at least $2c$ leaves.
\end{enumerate}
\end{proposition}

We now give two definitions that are not as standard as the others we have given so far.  We will use these definitions in Section 5.  For further reading on these concepts see \cite{bHA}. 

\begin{definition}
The {\bf weight} of a vertex $v$ of a tree $T$ is the maximal number of edges in any subtree of $T$ containing $v$ as a leaf.
\end{definition}
\begin{definition}
The {\bf centroid} of a tree $T$ is the set of all vertices of $T$ having minimum weight.
\end{definition}
An example of the weights of vertices of a tree is shown in Figure \ref{fig:onecentroid}. In that graph, the vertex with weight $8$ is the centroid of the tree.

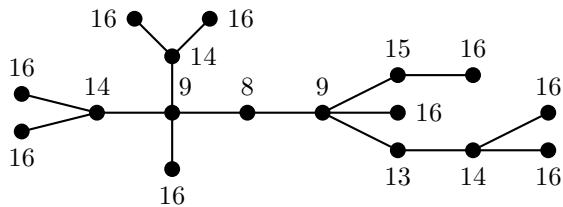
\begin{figure}[ht]
\centering
\begin{tikzpicture}[style=thick]

\draw (3, .25) coordinate (A14) node[below=3pt] {$16$};
\draw (6, .5) coordinate (A15) node[below=3pt] {$13$};
\draw (7, .5) coordinate (A16) node[below=3pt] {$14$};
\draw (8, .5) coordinate (A17) node[below=3pt] {$16$};

\draw (1, .75) coordinate (A7) node[below=3pt] {$16$};
\draw (2, 1) coordinate (A8) node[above=3pt] {$14$};
\draw (3, 1) coordinate (A9) node[above=3pt] {$\ \ \ 9$};

\draw (4, 1) coordinate (A10) node[above=3pt] {$8$};
\draw (5, 1) coordinate (A11) node[above=3pt] {$9$};
\draw (6, 1) coordinate (A12) node[right=3pt] {$16$};
\draw (8, 1) coordinate (A13) node[above=3pt] {$16$};

\draw (1, 1.25) coordinate (A3) node[above=3pt] {$16$};
\draw (3, 1.75) coordinate (A4) node[right=3pt] {$14$};
\draw (6, 1.5) coordinate (A5) node[above=3pt] {$15$};
\draw (7, 1.5) coordinate (A6) node[above=3pt] {$16$};

\draw (2.5, 2.25) coordinate (A1) node[left=3pt] {$16$};
\draw (3.5, 2.25) coordinate (A2) node[right=3pt] {$16$};

\draw(A1) -- (A4);
\draw(A2) -- (A4);

\draw(A3) -- (A8);
\draw(A4) -- (A9);
\draw(A5) -- (A11);
\draw(A6) -- (A5);

\draw(A7) -- (A8);
\draw(A8) -- (A9);
\draw(A9) -- (A10);
\draw(A10) -- (A11);
\draw(A11) -- (A12);
\draw(A9) -- (A14);
\draw(A11) -- (A15);
\draw(A13) -- (A16);
\draw(A15) -- (A16);
\draw(A16) -- (A17);

  \fill (A1) circle (3pt) 
  (A2) circle (3pt)
  (A3) circle (3pt)
  (A4) circle (3pt)
  (A5) circle (3pt)
  (A6) circle (3pt)
  (A7) circle (3pt)
  (A8) circle (3pt)
  (A9) circle (3pt)
  (A10) circle (3pt)
  (A11) circle (3pt)        
  (A12) circle (3pt)
  (A13) circle (3pt)        
  (A14) circle (3pt)
  (A15) circle (3pt)        
  (A16) circle (3pt)
  (A17) circle (3pt);

\end{tikzpicture}
\caption{A tree with one centroid}
\label{fig:onecentroid}
\end{figure}

\begin{proposition}[\cite{bBo}, pp. 99]
Every tree has a centroid consisting of either one vertex or two adjacent vertices.
\end{proposition}

\begin{figure}[ht]
\centering
\begin{tikzpicture}[style=thick]

\draw (3, .25) coordinate (A14) node[below=3pt] {$15$};
\draw (5, .5) coordinate (A15) node[below=3pt] {$12$};
\draw (6, .5) coordinate (A16) node[below=3pt] {$13$};
\draw (7, .5) coordinate (A17) node[below=3pt] {$15$};

\draw (1, .75) coordinate (A7) node[below=3pt] {$15$};
\draw (2, 1) coordinate (A8) node[above=3pt] {$13$};
\draw (3, 1) coordinate (A9) node[above=3pt] {$\ \ \ 8$};
\draw (4, 1) coordinate (A11) node[above=3pt] {$8$};
\draw (5, 1) coordinate (A12) node[right=3pt] {$15$};
\draw (7, 1) coordinate (A13) node[above=3pt] {$15$};

\draw (1, 1.25) coordinate (A3) node[above=3pt] {$15$};
\draw (3, 1.75) coordinate (A4) node[right=3pt] {$13$};
\draw (5, 1.5) coordinate (A5) node[above=3pt] {$14$};
\draw (6, 1.5) coordinate (A6) node[above=3pt] {$15$};

\draw (2.5, 2.25) coordinate (A1) node[left=3pt] {$15$};
\draw (3.5, 2.25) coordinate (A2) node[right=3pt] {$15$};

\draw(A1) -- (A4);
\draw(A2) -- (A4);

\draw(A3) -- (A8);
\draw(A4) -- (A9);
\draw(A5) -- (A11);
\draw(A6) -- (A5);

\draw(A7) -- (A8);
\draw(A8) -- (A9);
\draw(A9) -- (A11);
\draw(A11) -- (A12);
\draw(A9) -- (A14);
\draw(A11) -- (A15);
\draw(A13) -- (A16);

\draw(A15) -- (A16);
\draw(A16) -- (A17);

  \fill (A1) circle (3pt) 
  (A2) circle (3pt)
  (A3) circle (3pt)
  (A4) circle (3pt)
  (A5) circle (3pt)
  (A6) circle (3pt)
  (A7) circle (3pt)
  (A8) circle (3pt)
  (A9) circle (3pt)
  (A10) circle (3pt)
  (A11) circle (3pt)        
  (A12) circle (3pt)
  (A13) circle (3pt)        
  (A14) circle (3pt)
  (A15) circle (3pt)        
  (A16) circle (3pt)
  (A17) circle (3pt);

\end{tikzpicture}
\caption{A tree with two centroids}
\label{fig:twocentroids}
\end{figure}
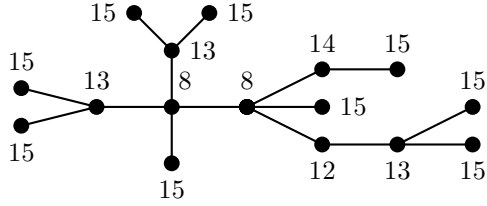

A \emph{proper coloring} of the vertices of a graph $G$ with $k$ colors is a function $\kappa: V(G) \rightarrow \{1, 2, \dots k\}$  such that $\kappa(v) \neq \kappa(w)$ for adjacent vertices $v, w \in V(G)$. Let $\mathbb{N}$ denote the positive integers.  The \emph{chromatic polynomial} of $G$ is a polynomial $\chi_G: \mathbb{N} \rightarrow \mathbb{N} \cup \{0\}$ having the property that $\chi_G(k)$ is the number of proper $k$-colorings of $G$.

For more properties of the chromatic polynomial, including the fact that it is a polynomial, see \cite{bHA}. The graph $G$ in Figure \ref{fig:chromfunct} has chromatic polynomial $\chi_G(k) = k(k - 1)^6$.  One remarkable feature of the chromatic polynomial is that for any fixed $n>0$, all trees with $n$ vertices have the same chromatic polynomial.

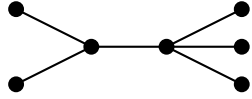
\begin{figure}[ht]
\centering
\begin{tikzpicture}[style=thick]

\draw (2, .5) coordinate (A6);
\draw (2, 1.5) coordinate (A1);
\draw (3, 1) coordinate (A3);
\draw (4, 1) coordinate (A4);
\draw (5, 1) coordinate (A5);
\draw (5, .5) coordinate (A7);
\draw (5, 1.5) coordinate (A2);

\draw(A1) -- (A3);
\draw(A2) -- (A4);
\draw(A3) -- (A4);
\draw(A3) -- (A6);
\draw(A4) -- (A5);
\draw(A4) -- (A7);

  \fill (A1) circle (3pt) 
  (A2) circle (3pt)
  (A3) circle (3pt)
  (A4) circle (3pt)
  (A5) circle (3pt) 
  (A6) circle (3pt)
  (A7) circle (3pt);

\end{tikzpicture}
\caption{A graph $G$ with chromatic polynomial $\chi_G(k) = k(k - 1)^6$}
\label{fig:chromfunct}
\end{figure}

\subsection{The chromatic symmetric function}
This paper will focus on the {\em chromatic symmetric function} of a graph. Before defining this function, we review basic facts of symmetric functions.

A \emph{partition} is a sequence $\lambda=(\lambda_1, \lambda_2, \ldots, \lambda_l)$ of positive integers such that $\lambda_1\geq \lambda_2\geq \cdots \geq \lambda_l$.   The $\lambda_i$ are called the \emph{parts} of $\lambda$.  Furthermore, we say that $\lambda$ is  a partition of $n$, written $\lambda\vdash n$, if $\sum_i \lambda_i= n$. 

Let $x_1, x_2, \ldots$ be a countably infinite set of commuting indeterminates.  For any positive integer $k$, define the \emph{power-sum symmetric function} as 
\[p_k =\sum_{i\geq 1} x_i^k\]
and for a partition $\lambda= ( \lambda_1, \lambda_2, \ldots, \lambda_l)$ we define
\[p_\lambda = p_{\lambda_1}p_{\lambda_2}\cdots p_{\lambda_l}\]
It is a well-known fact that $\{p_\lambda\, |\, \lambda\vdash n\}$ is a basis for the $\mathbb{Q}$-vector space $\Lambda_n$ of all symmetric functions that are homogeneous of degree $n$.  For more details about symmetric functions see \cite[Chap. 7]{bST3}. 

Let $G$ be a simple graph. Stanley \cite{bST}, see also \cite[pp. 462-464]{bST3}, defined the \emph{chromatic symmetric function} of $G$ as 
\begin{equation} \label{stan1}
{\bf X}_G = {\bf X}_G(x_1, x_2, \dots) = \sum_{\kappa}\prod_{v \in V(G)}x_{\kappa(v)}
\end{equation}
where the sum is over all proper colorings $\kappa$ and the $x_1, x_2, \ldots $ are a countably infinite set of commuting  indeterminates.   Since a coloring of a graph is invariant under permutation of the colors, ${\bf X}_G$ is a symmetric, homogeneous function of degree $\# V(G)$.    ${\bf X}_G$ is a generalization of the well-known single variable chromatic polynomial of a graph $\chi_G(k)$.   In fact, Stanley showed that if we set $x_1=x_2=\ldots = x_k=1$ and $x_i=0$ for all $i>k$, then 
\[{\bf X}_G(1,1,\ldots ,1, 0, \ldots ) = \chi_G(k).\]

Often when working with a symmetric function, it is helpful to expand it in terms of one of the many bases for the space of symmetric functions. 
\begin{theorem}{\textrm{[St1]}} 
\begin{equation} \label{stan}
{\bf X}_G = \sum_{S \subseteq E(G)}(-1)^{\#S}p_{\pi(S)}\textrm{,}
\end{equation}
where $\pi(S)$ is the partition whose parts are the orders of the connected components of the subgraphs of $G$ induced by $S$.  $\pi(S)$ is called the type of $S$, (see Figure \ref{fig:type} for an example of $\pi(S)$).
\end{theorem}

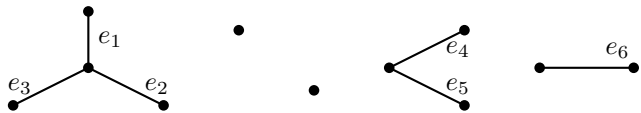
\begin{figure}[ht]
\centering
\begin{tikzpicture}[style=thick]

\draw (2, .5) coordinate (A3);
\draw (3, 1) coordinate (A2);
\draw (4, .5) coordinate (A4);
\draw (3, 1.75) coordinate (A1);
\draw(5, 1.5) coordinate(A5);
\draw(6,.7) coordinate(A6);

\draw (8, 1.5) coordinate (B1);
\draw (7, 1) coordinate (B2);
\draw (9, 1) coordinate (B3);
\draw (10.25, 1) coordinate (B4);
\draw (8, .5) coordinate (B5);

\draw(A1) -- (A2) node [midway] {$\ \ \quad e_1$};
\draw(A2) -- (A4)node [midway]   {$\ \ \ \ \quad e_2$};
\draw(A2) -- (A3)node [midway]   {$e_3$\qquad \ \ };

\draw(B1) -- (B2) node [midway] {$\ \ \ \ \quad e_4$};
\draw(B5) -- (B2)  node [midway] {$\ \ \ \ \quad e_5$};;
\draw(B4) -- (B3)  node [midway,above] {$\ \ \ \ \quad e_6$};;

  \fill (A1) circle (2pt) 
  (A2) circle (2pt)
  (A3) circle (2pt)
  (A4) circle (2pt)
  (A5) circle(2pt)
  (A6) circle (2pt)
  (B1) circle (2pt) 
  (B2) circle (2pt)
  (B3) circle (2pt)
  (B4) circle (2pt)
  (B5) circle (2pt);
\end{tikzpicture}
\caption{$S=\{e_1, e_2, e_3, e_4,e_5,e_6\}$ is an edge set of type $\pi(S)=(4,3,2,1,1)$}
\label{fig:type}
\end{figure}

There are several properties of a graph that we can determine from its chromatic symmetric function. In the following sections, we explore these properties.


\section{Properties preserved by the chromatic symmetric function}
We begin by studying the properties of a simple graph that are determined by its chromatic symmetric function.  We review some known results and make two new contributions. In particular, we show that the number of triangles and the sum of the squares of the vertex degrees can be recovered from ${\bf X}_G$.   The results in this section will be applicable to later sections, when we restrict our attention to unicyclic graphs and trees.

\begin{remark}\label{R1}
If two graphs have the same chromatic symmetric function, then they have
\begin{enumerate}
\item[(1)] the same number of edges, 
\item[(2)] the same number of vertices, and 
\item[(3)] the same number of matchings of $k$ edges (for any natural number $k$).
\end{enumerate}
\end{remark}
From Equation (\ref{stan}) the number of vertices can be recovered from the coefficient of $p_{(1,1,\dots,1)}$, the number of edges from the coefficient of  $p_{(2,1,\dots,1)}$ and the number of $k$-matchings from the coefficient of $p_{(2,2, \dots ,2,1,\dots, 1)}$ (containing $k$ 2's).

Fougere  \cite[Theorem 3.3.1]{bFo} proved that the sum of the squared vertex degrees of a tree was determined by its chromatic symmetric function. This result was strengthened by Martin, Morin, and Wagner when they showed that the degree and path sequences of a tree can be recovered from its chromatic symmetric function \cite[Corollary 5]{bMMA}. In the following proposition, we show that Fougere's result is true even for general graphs. However, the strengthened result in \cite[Corollary 5]{bMMA} is not true for a general graph. An example of two graphs with the same chromatic symmetric function yet differing degree sequences is given in Figure \ref{fig:unic}.
\begin{proposition}\label{degree_squared}
The sum of the squared vertex degrees of a graph $G$
$$\sum_{v \in V(G)}d(v)^2$$
can be obtained from $\textbf{X}_G$.
\end{proposition}

\begin{proof}
Let $S_G^{(2,2)}$ denote the number of spanning subgraphs of $G$ consisting of two disjoint edges and $\#V(G) - 4$ isolated vertices. Since there are no other spanning subgraphs having vertex partition $(2, 2, 1, \dots, 1)$, it follows that $S_G^{(2,2)}$ is exactly the coefficient of $p_{(2, 2, 1, \dots, 1)}$ in $\textbf{X}_G$. Therefore, if $G$ and $H$ have the same chromatic symmetric function, then $S_G^{(2,2)} = S_H^{(2,2)}$. Next, let $S_G^{(3)}$ be the number of spanning subgraphs of $G$ consisting of two non-disjoint edges and $\#V(G) - 3$ isolated vertices. Since all 2-edge spanning subgraphs of $G$ consist of those tallied in $S_G^{(2, 2)}$ or $S_G^{(3)}$, it follows that 

\[
\binomi{\#E(G)}{2} = S_G^{(2, 2)} + S_G^{(3)}\textrm{.}
\]

For graphs $G$ and $H$ having the same chromatic symmetric function $\#E(G) = \#E(H)$ by Remark \ref{R1} and $S_G^{(2, 2)} = S_H^{(2, 2)}$ for the reasons above. Therefore, $S_G^{(3)} = S_H^{(3)}$. We can count $S_G^{(3)}$ by noting that any pair of non-disjoint edges is uniquely determined by its central vertex and a choice of two edges incident to that vertex. Using the fact that $G$ and $H$ contain the same number of edges (from Remark \ref{R1}) and therefore must have the same degree sum, we calculate
\begin{align*}
\sum_{v \in V(G)}\binomi{d(v)}{2} &= \sum_{v \in V(H)}\binomi{d(v)}{2} \\
\sum_{v \in V(G)}(d(v)^2 - d(v)) &= \sum_{v \in V(H)}(d(v)^2 - d(v)) \\
\sum_{v \in V(G)}d(v)^2 &= \sum_{v \in V(H)}d(v)^2 \\
\end{align*}
\end{proof}

\begin{corollary}\label{triangles}
The number of triangles $T_G$ in a graph $G$ can be obtained from $\textbf{X}_G$.
\end{corollary}
\begin{proof}
Note that the coefficient of $p_{(3, 1, 1, \dots, 1)}$ in ${\bf X}_G$ consists of $S_G^{(3)} - T_G$. In the proof of Proposition \ref{degree_squared}, we saw that $S_G^{(3)}$ can be obtained from $\textbf{X}_G$. Therefore, $T_G$ can also be obtained from $\textbf{X}_G$.
\end{proof}

Chow \cite{bCh} proved that the planarity of a graph is not determined by its chromatic symmetric function. Also, in \cite[Proposition 3]{bMMA}, the authors show that the chromatic symmetric function of a graph determines its girth.

\section{Decomposition techniques for graphs}

A fundamental property of chromatic polynomials, $\chi_G(k)$, is the \emph{deletion-contraction} property.  If $G \backslash e$ denotes $G$ with edge $e$ deleted and $G/e$ denotes $G$ with edge $e$ contracted to a point, then 
\[\chi_G(k) = \chi_{G\backslash e} (k) - \chi_{G/e}(k).\]
This property is often used to prove other properties of $\chi_G$ using induction on the number of edges. Unfortunately, there is no analogous deletion-contraction property for ${\bf X}_G$.   In this section we present a novel technique for writing the chromatic symmetric function of a graph  as a linear combination of the chromatic symmetric function of other graphs. In the case that $G$ has girth three we are able to write $\X_G$ as a linear combination of chromatic symmetric functions of graphs with fewer edges. 

\begin{theorem}\label{erase_triangle}
Let $G$ be a graph where $e_1, e_2, e_3 \in E(G)$ form a triangle. Furthermore, define
\begin{itemize}
\item $G_{2, 3} = (V(G), E(G) - \{e_1\})$
\item $G_{1, 3} = (V(G), E(G) - \{e_2\})$
\item $G_{3} = (V(G), E(G) - \{e_1, e_2\})$
\end{itemize}
Then
$${\bf X}_G = {\bf X}_{G_{2, 3}} + {\bf X}_{G_{1, 3}} - {\bf X}_{G_{3}}\textrm{.}$$
\end{theorem}
\begin{proof}
Consider the following partition of the set of spanning subgraphs of $G$:
\begin{itemize}
\item $G^1 = \{S \subseteq E(G) : e_1, e_2, e_3 \in S\}$
\item $G^2 = \{S \subseteq E(G) : e_1, e_2 \in S, e_3 \notin S\}$
\item $G^3 = \{S \subseteq E(G) : e_1, e_3 \in S, e_2 \notin S\}$
\item $G^4 = \{S \subseteq E(G) : e_2, e_3 \in S, e_1 \notin S\}$
\item $G^5 = \{S \subseteq E(G) : e_1 \in S, e_2, e_3 \notin S\}$
\item $G^6 = \{S \subseteq E(G) : e_2 \in S, e_1, e_3 \notin S\}$
\item $G^7 = \{S \subseteq E(G) : e_3 \in S, e_1, e_2 \notin S\}$
\item $G^8 = \{S \subseteq E(G) : e_1, e_2, e_3 \notin S\}$.
\end{itemize}

Then by Equation (\ref{stan})
\begin{align*}
{\bf X}_G &= \sum_{S \subseteq E(G)}(-1)^{|S|}p_{\pi(S)}\\
&= \sum_{i = 1}^{8}\sum_{S \in G^i}(-1)^{|S|}p_{\pi(S)}\\
&= \sum_{i \in \{4,6,7,8\}}\sum_{S \in G^i}(-1)^{|S|}p_{\pi(S)} + \sum_{i \in \{3,5,7,8\}}\sum_{S \in G^i}(-1)^{|S|}p_{\pi(S)}\\
& \ \ \ \ - \sum_{i \in \{7,8\}}\sum_{S \in G^i}(-1)^{|S|}p_{\pi(S)} + \sum_{i \in \{1,2\}}\sum_{S \in G^i}(-1)^{|S|}p_{\pi(S)}\\
&= {\bf X}_{G_{2, 3}} + {\bf X}_{G_{1, 3}} - {\bf X}_{G_{3}} + \sum_{i \in \{1,2\}}\sum_{S \in G^i}(-1)^{|S|}p_{\pi(S)}\textrm{.}
\end{align*}
It suffices to prove that the final term is equal to zero. Note, however, that for every subgraph in $G^2$, one can add the edge $e_3$ to get a corresponding subgraph in $G^1$. The addition of $e_3$ will not change the vertex partition of the subgraph in $G^2$, although it will add one edge. Thus, for every subgraph in $G^1$, there is a corresponding subgraph in $G^2$ with an opposite contribution to the chromatic symmetric function. Therefore,
$$\sum_{i \in \{1,2\}}\sum_{S \in G^i}(-1)^{|S|}p_{\pi(S)} = 0$$
and the proof is complete.
\end{proof}

To apply Theorem \ref{erase_triangle}, we introduce the equivalence relation $\sim_{\bf X}$ on linear combinations of graphs. Let $\{G_i\}_{i \leq p}$ and $\{H_i\}_{i \leq k}$ be sets of graphs, and let $\{c_i\}_{i \leq p}$ and $\{d_i\}_{i \leq k}$ be real numbers. We say that
\[
\sum_{i \leq k}c_iG_i \sim_{\bf X} \sum_{i \leq k}d_iH_i
\]
if 
\[
\sum_{i \leq k}c_i{\bf X}_{G_i} = \sum_{i \leq k}d_i{\bf X}_{H_i}\textrm{.}
\]
The following example illustrates how Theorem \ref{erase_triangle} can be applied to two graphs. The vertices of the graphs are labeled to correspond to the labels used in Theorem \ref{erase_triangle}.

\begin{figure}[ht]
\centering
\begin{tikzpicture}[style=thick] 

\draw (.5, 1) coordinate (A1);
\draw (1.5, 1) coordinate (A2);
\draw (2.5, 1) coordinate (A3);
\draw (.5, 2) coordinate (A4);
\draw (1.5, 2) coordinate (A5);
\draw (2.5, 2) coordinate (A6);

\draw (4.1, 1) coordinate (B1);
\draw (5.1, 1) coordinate (B2);
\draw (6.1, 1) coordinate (B3);
\draw (4.1, 2) coordinate (B4);
\draw (5.1, 2) coordinate (B5);
\draw (6.1, 2) coordinate (B6);

\draw (7.5, 1) coordinate (C1);
\draw (8.5, 1) coordinate (C2);
\draw (9.5, 1) coordinate (C3);
\draw (7.5, 2) coordinate (C4);
\draw (8.5, 2) coordinate (C5);
\draw (9.5, 2) coordinate (C6);

\draw (10.8, 1) coordinate (D1);
\draw (11.8, 1) coordinate (D2);
\draw (12.8, 1) coordinate (D3);
\draw (10.8, 2) coordinate (D4);
\draw (11.8, 2) coordinate (D5);
\draw (12.8, 2) coordinate (D6);

\draw(A1) -- (A2) node [midway, below] {$e_2$};
\draw(A3) -- (A2);
\draw(A5) -- (A2) node [midway, right] {$e_3$};
\draw(A5) -- (A1) node [midway, above] {$e_1$};
\draw(A4) -- (A1);
\draw(A5) -- (A6);

\draw(B3) -- (B2);
\draw(B5) -- (B2) node [midway, right] {$e_3$};
\draw(B5) -- (B1) node [midway, above] {$e_1$};
\draw(B4) -- (B1);
\draw(B5) -- (B6);

\draw(C1) -- (C2) node [midway, below] {$e_2$};
\draw(C3) -- (C2);
\draw(C5) -- (C2) node [midway, right] {$e_3$};
\draw(C4) -- (C1);
\draw(C5) -- (C6);

\draw(D3) -- (D2);
\draw(D5) -- (D2) node [midway, right] {$e_3$};
\draw(D4) -- (D1);
\draw(D5) -- (D6);

\draw (2.85, 1.5) coordinate (RA) node[right] { $\sim_{\bf X}$ };
\draw (6.4, 1.5) coordinate (PL) node[right] { \bf{$+$} };
\draw (9.9, 1.5) coordinate (MI) node[right] { \bf{$-$} };

  \fill (A1) circle (3pt) (A2) circle (3pt) (A3) circle (3pt)
    (A4) circle (3pt) (A5) circle (3pt) (A6) circle (3pt)
(B1) circle (3pt) (B2) circle (3pt) (B3) circle (3pt)
    (B4) circle (3pt) (B5) circle (3pt) (B6) circle (3pt)
(C1) circle (3pt) (C2) circle (3pt) (C3) circle (3pt)
    (C4) circle (3pt) (C5) circle (3pt) (C6) circle (3pt)
(D1) circle (3pt) (D2) circle (3pt) (D3) circle (3pt)
    (D4) circle (3pt) (D5) circle (3pt) (D6) circle (3pt);
\end{tikzpicture}
\label{fig:trierase1}
\end{figure}

\begin{figure}[ht]
\centering
\begin{tikzpicture}[style=thick] 

\draw (.1, 1) coordinate (A1);
\draw (1.1, 1) coordinate (A2);
\draw (2.1, 1) coordinate (A3);
\draw (.1, 2) coordinate (A4);
\draw (1.1, 2) coordinate (A5);
\draw (2.1, 2) coordinate (A6);

\draw (3.9, 1) coordinate (B1);
\draw (4.9, 1) coordinate (B2);
\draw (5.9, 1) coordinate (B3);
\draw (3.9, 2) coordinate (B4);
\draw (4.9, 2) coordinate (B5);
\draw (5.9, 2) coordinate (B6);

\draw (7.2, 1) coordinate (C1);
\draw (8.2, 1) coordinate (C2);
\draw (9.2, 1) coordinate (C3);
\draw (7.2, 2) coordinate (C4);
\draw (8.2, 2) coordinate (C5);
\draw (9.2, 2) coordinate (C6);

\draw (10.5, 1) coordinate (D1);
\draw (11.5, 1) coordinate (D2);
\draw (12.5, 1) coordinate (D3);
\draw (10.5, 2) coordinate (D4);
\draw (11.5, 2) coordinate (D5);
\draw (12.5, 2) coordinate (D6);

\draw(A1) -- (A2) node [midway, below] {$e_2$};
\draw(A3) -- (A2);
\draw(A5) -- (A2);
\draw(A4) -- (A2) node [midway, above] {$e_1$};
\draw(A4) -- (A1) node [midway, left] {$e_3$};
\draw(A5) -- (A6);

\draw(B3) -- (B2);
\draw(B5) -- (B2);
\draw(B4) -- (B2) node [midway, above] {$e_1$};
\draw(B4) -- (B1)  node [midway] {$e_3 \ \ \ $};
\draw(B5) -- (B6);

\draw(C1) -- (C2) node [midway, below] {$e_2$};
\draw(C3) -- (C2);
\draw(C5) -- (C2);
\draw(C4) -- (C1)  node [midway, right] {$e_3$};
\draw(C5) -- (C6);

\draw(D3) -- (D2);
\draw(D5) -- (D2);
\draw(D4) -- (D1)  node [midway, right] {$e_3$};
\draw(D5) -- (D6);

\draw (2.5, 1.5) coordinate (RA) node[right] { $\sim_{\bf X}$ };
\draw (6.3, 1.5) coordinate (PL) node[right] { $\bf{+} $};
\draw (9.5, 1.5) coordinate (MI) node[right] {$ \bf{-}$ };

  \fill (A1) circle (3pt) (A2) circle (3pt) (A3) circle (3pt)
    (A4) circle (3pt) (A5) circle (3pt) (A6) circle (3pt)
(B1) circle (3pt) (B2) circle (3pt) (B3) circle (3pt)
    (B4) circle (3pt) (B5) circle (3pt) (B6) circle (3pt)
(C1) circle (3pt) (C2) circle (3pt) (C3) circle (3pt)
    (C4) circle (3pt) (C5) circle (3pt) (C6) circle (3pt)
(D1) circle (3pt) (D2) circle (3pt) (D3) circle (3pt)
    (D4) circle (3pt) (D5) circle (3pt) (D6) circle (3pt);
\end{tikzpicture}
\caption{Decomposition of Two Graphs}
\label{fig:trierase2}
\end{figure}
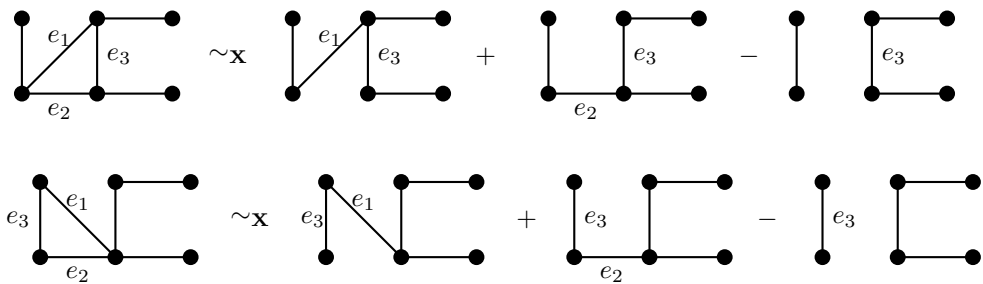
Notice that the forests in the top expression are isomorphic to the forests in the bottom expression. Therefore, the unicyclic graph on the top left in Figure \ref{fig:trierase2} has the same chromatic symmetric function as the unicyclic graph on the bottom left of Figure \ref{fig:trierase2}. This example shows that the degree sequence of a graph is not determined by its chromatic symmetric function.

However, this method has limitations. At this point in our discussion, we cannot apply it to graphs having no triangles. Also, even if we do manage to decompose the chromatic symmetric function of two graphs into linear combinations of chromatic symmetric functions of forests, we cannot always determine whether these linear combinations are equal. The following corollary introduces another way to decompose chromatic symmetric functions that will complement Theorem \ref{erase_triangle} to strengthen the method.

Notice that there are six different ways to label the three sides of a triangle $e_1, e_2,$ and $e_3$. By choosing different permutations of labels, Theorem \ref{erase_triangle} yields three different ways to decompose a graph with a triangle into smaller graphs. We use this fact in the following corollary.

\begin{corollary} \label{line_erase}
Let $G$ be a graph with the adjacent edges $e_1 = vv_1$, $e_2 = vv_2$ and $e_3 = v_1v_2 \notin E(G)$ (that is, $e_1$ and $e_2$ meet at the vertex $v$, but there is no edge connecting $v_1$ to $v_2$). Define
\begin{itemize}
\item $G_{1, 3} = (V(G), (E(G) - \{e_2\}) \cup \{e_3\})$
\item $G_{2, 3} = (V(G), (E(G) - \{e_1\}) \cup \{e_3\})$
\item $G_{1} = (V(G), E(G) - \{e_2\})$
\item $G_{3} = (V(G), (E(G) - \{e_1, e_2\}) \cup \{e_3\})\textrm{.}$
\end{itemize}
Then 
$${\bf X}_G = {\bf X}_{G_{2, 3}} + {\bf X}_{G_{1}} - {\bf X}_{G_3}\textrm{.}$$
\end{corollary}
\begin{proof}
Let $G^\prime = (V(G), E(G) \cup \{e_3\})$. We can apply Theorem \ref{erase_triangle} in two different ways:
$${\bf X}_{G^\prime} = {\bf X}_{G_{2, 3}} + {\bf X}_{G_{1, 3}} - {\bf X}_{G_{3}}$$
$${\bf X}_{G^\prime} = {\bf X}_{G_{1, 3}} + {\bf X}_G - {\bf X}_{G_{1}}\textrm{.}$$
This gives the desired equality.
\end{proof}

Figure \ref{fig:rule2} illustrates Corollary \ref{line_erase}, and Figure \ref{fig:twist} shows an application of it.

\begin{figure}[ht]
\centering
\begin{tikzpicture}[style=thick] 

\draw (1, 1) coordinate (A2);
\draw (2, 1) coordinate (A3);
\draw (1.5, 2) coordinate (A1) node[right] {$v$};

\draw (3.5, 1) coordinate (B2);
\draw (4.5, 1) coordinate (B3);
\draw (4, 2) coordinate (B1);

\draw (5.5, 1) coordinate (C2);
\draw (6.5, 1) coordinate (C3);
\draw (6, 2) coordinate (C1);

\draw (7.5, 1) coordinate (D2);
\draw (8.5, 1) coordinate (D3);
\draw (8, 2) coordinate (D1);

\draw (9.5, 1) coordinate (E2);
\draw (10.5, 1) coordinate (E3);
\draw (10, 2) coordinate (E1);

\draw(A1) -- (A2) node [midway] {$e_1 \ \ \ \ $};
\draw(A3) -- (A1) node [midway] {$\ \ \ e_2$};

\draw(B1) -- (B3);
\draw(B3) -- (B2);

\draw(C2) -- (C1);

\draw(D3) -- (D2);

\draw (2.2, 1.5) coordinate (RA) node[right] { $\sim_{\bf X}$ };
\draw (4.7, 1.5) coordinate (PL) node[right] { \bf{+} };
\draw (6.7, 1.5) coordinate (MI) node[right] { \bf{-} };

  \fill (A1) circle (3pt) (A2) circle (3pt) (A3) circle (3pt)
(B1) circle (3pt) (B2) circle (3pt) (B3) circle (3pt)
(C1) circle (3pt) (C2) circle (3pt) (C3) circle (3pt)
(D1) circle (3pt) (D2) circle (3pt) (D3) circle (3pt);

\end{tikzpicture}
\caption{Illustration of Corollary \ref{line_erase}}
\label{fig:rule2}
\end{figure}
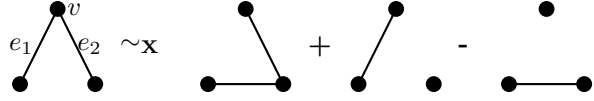

\begin{figure}[ht]
\centering
\begin{tikzpicture}[style=thick] 

\draw (1, 1) coordinate (A1);
\draw (2, 1) coordinate (A2);
\draw (2.309, 1.951) coordinate (A3);
\draw (.691, 1.951) coordinate (A4);
\draw (1.5, 2.64) coordinate (A5);

\draw (4, 1) coordinate (B1);
\draw (5, 1) coordinate (B2);
\draw (5.309, 1.951) coordinate (B3);
\draw (3.691, 1.951) coordinate (B4);
\draw (4.5, 2.64) coordinate (B5);

\draw (7, 1) coordinate (C1);
\draw (8, 1) coordinate (C2);
\draw (8.309, 1.951) coordinate (C3);
\draw (6.691, 1.951) coordinate (C4);
\draw (7.5, 2.64) coordinate (C5);

\draw (10, 1) coordinate (D1);
\draw (11, 1) coordinate (D2);
\draw (11.309, 1.951) coordinate (D3);
\draw (9.691, 1.951) coordinate (D4);
\draw (10.5, 2.64) coordinate (D5);

\draw(A1) -- (A2);
\draw(A3) -- (A2);
\draw(A5) -- (A3) node [midway, above] {$e_2$};
\draw(A4) -- (A5) node [midway, above] {$e_1$};
\draw(A4) -- (A1);

\draw(B1) -- (B2);
\draw(B1) -- (B4);
\draw(B2) -- (B3);
\draw(B4) -- (B3)  node [midway, below] {$e_3$};
\draw(B5) -- (B3)  node [midway, above] {$e_2$};

\draw(C1) -- (C2);
\draw(C3) -- (C2);
\draw(C1) -- (C4);
\draw(C4) -- (C5)  node [midway, above] {$e_1$};

\draw(D1) -- (D2);
\draw(D3) -- (D2);
\draw(D1) -- (D4);
\draw(D4) -- (D3)  node [midway, below] {$e_3$};

\draw (2.7, 1.5) coordinate (RA) node[right] { $\sim_{\bf X}$ };
\draw (5.9, 1.5) coordinate (PL) node[right] { \bf{+} };
\draw (8.9, 1.5) coordinate (MI) node[right] { \bf{-} };

  \fill (A1) circle (3pt) (A2) circle (3pt) (A3) circle (3pt)
    (A4) circle (3pt) (A5) circle (3pt) 
(B1) circle (3pt) (B2) circle (3pt) (B3) circle (3pt)
    (B4) circle (3pt) (B5) circle (3pt) 
(C1) circle (3pt) (C2) circle (3pt) (C3) circle (3pt)
    (C4) circle (3pt) (C5) circle (3pt) 
(D1) circle (3pt) (D2) circle (3pt) (D3) circle (3pt)
    (D4) circle (3pt) (D5) circle (3pt) ;
\end{tikzpicture}
\caption{Application of Corollary \ref{line_erase}}
\label{fig:twist}
\end{figure}
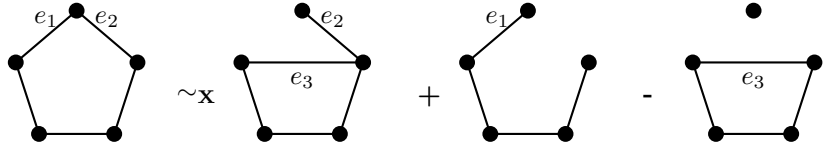
There is one more corollary to Theorem \ref{erase_triangle} that will be useful.

\begin{corollary} \label{wedge_erase}
Let $G$ be a graph contaning a triangle formed by the edges $e_1 = vv_1$, $e_2 = vv_2$ and $e_3 = v_1v_2$. Define
\begin{itemize}
\item $G_{1, 2} = (V(G), E(G) - \{e_3\})$
\item $G_{1} = (V(G), E(G) - \{e_2, e_3\})$
\item $G_{2} = (V(G), E(G) - \{e_1, e_3\})$
\item $G_{3} = (V(G), E(G) - \{e_1, e_2\})\textrm{.}$
\end{itemize}
Then 
$${\bf X}_G = 2{\bf X}_{G_{1, 2}} + {\bf X}_{G_{3}} - {\bf X}_{G_1} - {\bf X}_{G_2}\textrm{.}$$
\end{corollary}
\begin{proof}
If we define
$$G_{2, 3} = (V(G), E(G) - \{e_1\})\textrm{,}$$
then we can apply Theorem \ref{erase_triangle} and Corollary \ref{line_erase} to obtain the following equations:
$${\bf X}_{G} = {\bf X}_{G_{1, 2}} + {\bf X}_{G_{2, 3}} - {\bf X}_{G_{2}}$$
$${\bf X}_{G_{2, 3}} = {\bf X}_{G_{1, 2}} + {\bf X}_{G_3} - {\bf X}_{G_{1}}\textrm{.}$$
The desired equality follows from substitution of the second equation into the first.
\end{proof}

Corollary \ref{wedge_erase} is illustrated in Figure \ref{fig:rule1}.

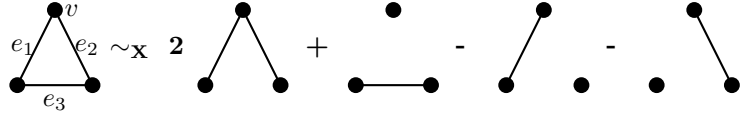
\begin{figure}[ht]
\centering
\begin{tikzpicture}[style=thick]

\draw (1, 1) coordinate (A2);
\draw (2, 1) coordinate (A3);
\draw (1.5, 2) coordinate (A1) node[right] {$v$};

\draw (3.5, 1) coordinate (B2);
\draw (4.5, 1) coordinate (B3);
\draw (4, 2) coordinate (B1);

\draw (5.5, 1) coordinate (C2);
\draw (6.5, 1) coordinate (C3);
\draw (6, 2) coordinate (C1);

\draw (7.5, 1) coordinate (D2);
\draw (8.5, 1) coordinate (D3);
\draw (8, 2) coordinate (D1);

\draw (9.5, 1) coordinate (E2);
\draw (10.5, 1) coordinate (E3);
\draw (10, 2) coordinate (E1);

\draw(A1) -- (A2) node [midway] {$e_1 \ \ \ $};
\draw(A3) -- (A1) node [midway] {$\ \ \ e_2$};
\draw(A3) -- (A2) node [midway, below] {$e_3$};

\draw(B3) -- (B1);
\draw(B1) -- (B2);

\draw(C3) -- (C2);

\draw(D1) -- (D2);

\draw(E3) -- (E1);

\draw (2.1, 1.5) coordinate (RA) node[right] { $\sim_{\bf X}$ \bf{ 2} };
\draw (4.7, 1.5) coordinate (PL) node[right] { \bf{+} };
\draw (6.7, 1.5) coordinate (MI) node[right] { \bf{-} };
\draw (8.7, 1.5) coordinate (MI2) node[right] { \bf{-} };

  \fill (A1) circle (3pt) (A2) circle (3pt) (A3) circle (3pt)
(B1) circle (3pt) (B2) circle (3pt) (B3) circle (3pt)
(C1) circle (3pt) (C2) circle (3pt) (C3) circle (3pt)
(D1) circle (3pt) (D2) circle (3pt) (D3) circle (3pt)
(E1) circle (3pt) (E2) circle (3pt) (E3) circle (3pt);
\end{tikzpicture}
\caption{Illustration of Corollary \ref{wedge_erase}}
\label{fig:rule1}
\end{figure}


\section{The chromatic symmetric function of a unicyclic graph}

In this section, we study chromatic symmetric functions of unicyclic graphs. In \cite{bMMA}, the authors describe two families of unicyclic graphs, {\em squids} and {\em crabs}. A {\em squid} is a connected unicyclic graph having only one vertex of degree greater than 2, while a {\em crab} is a connected unicyclic graph in which every vertex not lying on the cycle has degree 1. In \cite[Theorem 12]{bMMA}, they proved that no two non-isomorphic squids have the same chromatic symmetric function.  In \cite[Proposition 13]{bMMA} they show a similar result for crabs subject to an additional technical condition. 
In addition, Martin, Morin and Wagner asked whether two distinct unicyclic graphs could have the same chromatic symmetric function. Figure \ref{fig:unic} shows an example of two such graphs.

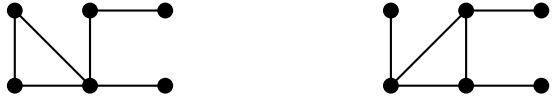
\begin{figure}[ht]
\centering
\begin{tikzpicture}[style=thick] 

\draw (2, .5) coordinate (C3) node[below=3pt] {};
\draw (3, .5) coordinate (C4) node[below=3pt] {};
\draw (2, 1.5) coordinate (C1) node[above=3pt] {};
\draw (3, 1.5) coordinate (C2) node[above=3pt] {};
\draw (4, .5) coordinate (D4) node {};
\draw (4, 1.5) coordinate (D2) node {};

\draw (7, .5) coordinate  (E3) node[below=3pt] {};
\draw (8, .5) coordinate (E4) node[below=3pt] {};
\draw (7, 1.5) coordinate  (E1) node[above=3pt] {};
\draw (8, 1.5) coordinate (E2) node[above=3pt] {};
\draw (9, .5) coordinate (U4) node {};
\draw (9, 1.5) coordinate (U2) node {};

\draw(C1) -- (C3);
\draw(C2) -- (C4);
\draw(C4) -- (C3);
\draw(C1) -- (C4);
\draw(C2) -- (D2);
\draw(C4) -- (D4);

\draw(E1) -- (E3);
\draw(E2) -- (E3);
\draw(E2) -- (E4);
\draw(E4) -- (E3);
\draw(E2) -- (U2);
\draw(E4) -- (U4);

\fill (C1) circle (3pt) (C2) circle (3pt) (C3) circle (3pt)
    (C4) circle (3pt) (D2) circle (3pt)
    (D4) circle (3pt);

\fill (E1) circle (3pt) (E2) circle (3pt) (E3) circle (3pt)
    (E4) circle (3pt) (U2) circle (3pt)
    (U4) circle (3pt);

\end{tikzpicture}
\caption{Two unicyclic graphs with the same chromatic symmetric function}
\label{fig:unic}
\end{figure}

These graphs are precisely the same ones we used to illustrate Theorem \ref{erase_triangle} in Section 3. It is also interesting that these graphs do {\em not} have the same degree sequence. It was proven in \cite{bMMA} that any two trees with the same chromatic symmetric function share the same degree sequence; here we see that an analogous result does not apply for unicyclic graphs. In fact, not even the number of leaves is determined from the chromatic symmetric function of a unicyclic graph. 

In the following proposition we provide a relation between the number of vertices of degree one and two within the cycle in two unicyclic graphs with the same chromatic symmetric function. 

\begin{proposition}
If graphs $G$ and $H$ are connected unicyclic graphs with a cycle of length $p$ and ${\bf X}_G = {\bf X}_H$, then
$$(p-1)L_G + I_G = (p-1)L_H + I_H$$
where $L_G$ and $L_H$ are the number of leaves in $G$ and $H$ respectively, and $I_G$ and $I_H$ are the number of vertices of degree two on the cycles of $G$ and $H$ respectively.
\end{proposition}

\begin{proof}
Let $n$ be the number of vertices in $G$ and $H$. Because $G$ and $H$ are unicyclic, the removal of any three edges of $G$ or $H$ will be a subgraph with at least three connected components. Therefore, every subgraph of $G$ or $H$ of type $(n-1,1)$ is a subgraph with either $n-1$ or $n-2$ edges. The only way that the removal of one edge would result in a type $(n-1,1)$ subgraph is if the removed edge is adjacent to a leaf. There are two ways that the removal of two edges could result in a type $(n-1,1)$ subgraph: either a edge adjacent to a leaf and a cycle edge are both removed, or two edges attached to a vertex of degree two on the cycle are removed. Thus, equating the coefficient of $p_{(n-1, 1)}$ in ${\bf X}_G$ and ${\bf X}_H$ yields
\[
(-1)^{n-1}L_G + (-1)^{n-2}pL_G + (-1)^{n-2}I_G = (-1)^{n-1}L_H + (-1)^{n-2}pL_H + (-1)^{n-2}I_H.
\]
Multiplying the above equation by $(-1)^{n}$ gives the desired result.
\end{proof}

So far, we have just one example of a pair of unicyclic graphs (those shown in Figure \ref{fig:unic}) that have the same chromatic symmetric function. Next, we will prove a theorem that gives a sufficient condition for two graphs to have the same chromatic symmetric function. It will aid us in constructing infinitely many pairs of unicyclic graphs with the same chromatic symmetric function. 

\begin{theorem}\label{P1}
Let $G = (V(G), E(G))$ be a graph that has four vertices $u$, $v$, $w$, $z$ with the property that $uz, wz, vw \in E(G)$ and $uw, vz, uv \notin E(G)$. If there exists a graph automorphism $\varphi: (V(G), E(G) - wz) \rightarrow (V(G), E(G) - wz)$ such that
$$\varphi(\{u, w\}) = \{v, z\} \ \ {\textrm{and}} \ \ \varphi(\{v, z\}) = \{u, w\}$$
then 
\[
H = (V(G), E(G) \cup \{uw\}) \textrm{ and } J = (V(G), E(G) \cup \{vz\})
\]
have the same chromatic symmetric function.
\end{theorem}

\begin{proof}
Consider the following partition of the spanning subgraphs of $H$:
\begin{itemize}
\item $H^1$ is the set of all spanning subgraphs of $H$ that do not contain the edge $uw$
\item $H^2$ is the set of all spanning subgraphs of $H$ that do contain the edge $uw$ but not $wz$
\item $H^3$ is the set of all spanning subgraphs of $H$ that contain the edges $uw$ and $wz$, but not $uz$
\item $H^4$ is the set of all spanning subgraphs of $H$ that contain the edges $uw$, $wz$, and $uz$
\end{itemize}
We partition the spanning subgraphs of $J$ similarly:
\begin{itemize}
\item $J^1$ is the set of all spanning subgraphs of $J$ that do not contain $vz$
\item $J^2$ is the set of all spanning subgraphs of $J$ that do contain $vz$ but not $wz$
\item $J^3$ is the set of all spanning subgraphs of $J$ that contain $vz$ and $wz$, but not $vw$
\item $J^4$ is the set of all spanning subgraphs of $J$ that contain $vz$, $wz$, and $vw$
\end{itemize}

First, note that $H^1$ can be put into a bijection with $J^1$ via the map of subgraphs induced by the identification $V(H) = V(J)$. Because this bijection preserves the vertex partition of a subgraph, the net contribution of $H^1$ to ${\bf X}_H$ equals the net contribution of $J^1$ to ${\bf X}_J$.

Next, note that $H^2$ can be put into a bijection with $J^2$ via the map of subgraphs induced by $\varphi$. Because this bijection preserves the vertex partition of a subgraph, the net contribution of $H^2$ to ${\bf X}_H$ equals the net contribution of $J^2$ to ${\bf X}_J$.

Note that the contribution of $H^3$ to ${\bf X}_H$ is the negative of the contribution of $H^4$ to ${\bf X}_H$. This is because the bijection from $H^3$ to $H^4$ that adds the edge $vw$ to each subgraph in $H^3$ preserves the vertex partition of the subgraph while adding an extra edge. Therefore, the contribution of subgraphs of type $H^3$ to ${\bf X}_H$ is exactly cancelled by the contribution of the subgraphs of type $H^4$ to ${\bf X}_H$. Similarly, the contribution of subgraphs of type $J^3$ to ${\bf X}_J$ is exactly cancelled by the contribution of the subgraphs of type $J^4$.

We see that the contribution of $H^1 \cup H^2 \cup H^3 \cup H^4$ to ${\bf X}_H$ is equal to the contribution of $J^1 \cup J^2 \cup J^3 \cup J^4$ to ${\bf X}_J$. Therefore, ${\bf X}_H = {\bf X}_J$.
\end{proof}

Notice that Theorem \ref{P1} can also be used to show that the two unicyclic graphs in Figure \ref{fig:unic} do indeed have the same chromatic symmetric function. In fact, Theorem \ref{P1} actually gives us infinitely many examples of unicyclic graphs having the same chromatic symmetric function. For any two nonisomorphic rooted trees $T_1$ and $T_2$, one can make two copies of each tree and connect the four trees according to the left picture in Figure \ref{fig:creatingident}. By Theorem \ref{P1}, this unicyclic graph will have the same chromatic symmetric function as the graph obtained using the connection rule in the right picture in Figure \ref{fig:creatingident}.

\begin{figure}[ht]
\centering
\begin{tikzpicture}[style=thick]

\draw (2, 1) coordinate (C3) node[below=3pt] {$z$};
\draw (4, 1) coordinate (C4) node[below=3pt] {$w$};
\draw (2, 3) coordinate (C1) node[above=3pt] {$u$};
\draw (4, 3) coordinate (C2) node[above=3pt] {$v$};
\draw (1, 1) coordinate (D3) node {\textbf{$T_1\quad$}};
\draw (1, 3) coordinate (D1) node {\textbf{$T_1\quad$}};
\draw (5, 1) coordinate (D4) node {\textbf{$\quad T_2$}};
\draw (5, 3) coordinate (D2) node {\textbf{$\quad T_2$}};

\draw (9, 1) coordinate  (E3) node[below=3pt] {$z$};
\draw (11, 1) coordinate (E4) node[below=3pt] {$w$};
\draw (9, 3) coordinate  (E1) node[above=3pt] {$u$};
\draw (11, 3) coordinate (E2) node[above=3pt] {$v$};
\draw (8, 1) coordinate  (U3) node {\textbf{$T_1\quad $}};
\draw (8, 3) coordinate  (U1) node {\textbf{$T_1\quad$}};
\draw (12, 1) coordinate (U4) node {\textbf{$\quad T_2$}};
\draw (12, 3) coordinate (U2) node {\textbf{$\quad T_2$}};

\draw(C1) -- (C3);
\draw(C2) -- (C4);
\draw(C4) -- (C3);
\draw(C1) -- (C4);
\draw[draw=black, dashed] (C3) -- (D3);
\draw[draw=black, dashed] (C1) -- (D1);
\draw[draw=black, dashed] (C2) -- (D2);
\draw[draw=black, dashed] (C4) -- (D4);

\draw(E1) -- (E3);
\draw(E2) -- (E3);
\draw(E2) -- (E4);
\draw(E4) -- (E3);
\draw[draw=black, dashed] (E3) -- (U3);
\draw[draw=black, dashed] (E1) -- (U1);
\draw[draw=black, dashed] (E2) -- (U2);
\draw[draw=black, dashed] (E4) -- (U4);

  \fill (C1) circle (3pt) (C2) circle (3pt) (C3) circle (3pt)
    (C4) circle (3pt);

\fill (E1) circle (3pt) (E2) circle (3pt) (E3) circle (3pt)
    (E4) circle (3pt);

\end{tikzpicture}
\caption{Creating unicyclic graphs with identical chromatic symmetric function}
\label{fig:creatingident}
\end{figure}
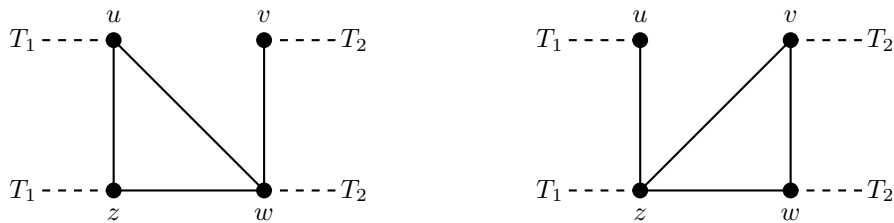

Although Theorem \ref{P1} was motivated by the question of whether unicyclic graphs could share a chromatic symmetric function, its usefulness is not limited to unicyclic graphs. On the contrary, one can use Theorem \ref{P1} to show that graphs such as those in Figure \ref{fig:nonuni} have the same chromatic symmetric function.

\begin{figure}[ht]
\centering
\begin{tikzpicture}[style=thick]

\draw (1, 1) coordinate (A6);
\draw (2, 1) coordinate (A7);
\draw (3, 1) coordinate (A8);
\draw (4, 1) coordinate (A9);
\draw (1, 2) coordinate (A1);
\draw (2, 2) coordinate (A2);
\draw (3, 2) coordinate (A3);
\draw (4, 2) coordinate (A4);
\draw (5, 1.5) coordinate (A5);

\draw (7, 1) coordinate (B6);
\draw (8, 1) coordinate (B7);
\draw (9, 1) coordinate (B8);
\draw (10, 1) coordinate (B9);
\draw (7, 2) coordinate (B1);
\draw (8, 2) coordinate (B2);
\draw (9, 2) coordinate (B3);
\draw (10, 2) coordinate (B4);
\draw (11, 1.5) coordinate (B5);

\draw(A1) -- (A2);
\draw(A2) -- (A3);
\draw(A4) -- (A5);
\draw(A6) -- (A7);
\draw(A7) -- (A8);
\draw(A8) -- (A9);
\draw(A9) -- (A5);
\draw(A1) -- (A6);
\draw(A2) -- (A7);
\draw(A3) -- (A8);
\draw(A4) -- (A9);
\draw(A1) -- (A7);
\draw(A2) -- (A6);

\draw(A3) -- (A9);

\draw(B1) -- (B2);
\draw(B2) -- (B3);
\draw(B4) -- (B5);
\draw(B6) -- (B7);
\draw(B7) -- (B8);
\draw(B8) -- (B9);
\draw(B9) -- (B5);
\draw(B1) -- (B6);
\draw(B2) -- (B7);
\draw(B3) -- (B8);
\draw(B4) -- (B9);
\draw(B1) -- (B7);
\draw(B2) -- (B6);

\draw(B4) -- (B8);

  \fill 
  (A1) circle (3pt)
  (A2) circle (3pt)
  (A3) circle (3pt)
  (A4) circle (3pt)
  (A5) circle (3pt)
  (A6) circle (3pt)
  (A7) circle (3pt)
  (A8) circle (3pt)
  (A9) circle (3pt)

  (B1) circle (3pt)
  (B2) circle (3pt)
  (B3) circle (3pt)
  (B4) circle (3pt)
  (B5) circle (3pt)
  (B6) circle (3pt)
  (B7) circle (3pt)
  (B8) circle (3pt)
  (B9) circle (3pt);
\end{tikzpicture}
\caption{Two non-isomorphic graphs with the same chromatic symmetric function}
\label{fig:nonuni}
\end{figure}
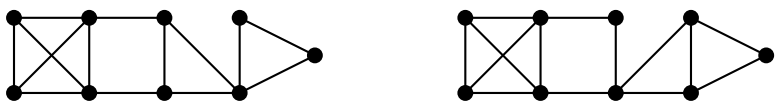


\section{Chromatic symmetric functions of trees}

In \cite{bST}, Stanley asks whether a tree $T$ is determined by ${\bf X}_T$. This question is still open. Throughout this section, we aim to prove results bringing us closer to answering this question. In particular, we will prove a classification theorem of trees related to Stanley's question.

Let $\mathcal{P}(X)$ denote the power set of a set $X$, and let Part($k$) denote the partitions of an integer $k$. For a tree $T$, we define the function
$$\theta_{T}: \mathcal{P}(E(T)) \rightarrow \textrm{Part}(\#V(T))$$
by $\theta_{T}(\{e_1, e_2, \dots, e_k\}) = \pi(E(T) - \{e_1, e_2, \dots, e_k\})$, where for $S \subseteq E(G)$, the partition $\pi(S)$ of $\#V(G)$ has as its parts the number of vertices in each connected component of the graph $(V(G), S)$. 

For positive integers $\{a_i\}_{1 \leq i \leq p}$ that sum to $n$, we define the rearrange function 
\[
\textrm{re}:(a_1, a_2, \dots, a_r)  \rightarrow \textrm{Part}(n)
\]
that sends a collection of $r$ integers to the partition having $a_1, a_2, \dots, a_r$ as its parts (recall that a partition is always written in weakly decreasing order). We define an ordering on the 2-part partitions of $n$ by $(n-i, i) > (n-j, j)$ if $i > j$.

Throughout this section, we will pay particular attention to the image under $\theta_T$ of singleton and 2-element sets of edges. Because $T$ is a tree, for any set $S$ of edges, $\theta_T(S)$ will have $\#S + 1$ parts. In particular, $\theta_T$ sends singleton sets to 2-part partitions of $\#V(T)$, and sends 2-element sets to 3-part partitions of $\#V(T)$. The following lemma can be easily proved by contradiction. 

\begin{lemma}  \label{centroid}
Let $T$ be a tree with even order $n$. Then there is at most one edge $e$ in $T$ such that $\theta_T(e) = (\frac{n}{2},\frac{n}{2})$.  If such an edge exists, it joins two centroids. 
\end{lemma} 

Hence, we have that we can tell from ${\bf X}_T$ if the tree $T$ has one or two centroids.  

\begin{proposition}\label{splits}
Let $T$ be a tree with  order $n$ having the distinct edges $e_a$ and $e_b$. Let $\theta_T(e_a) = (n-i, i)$ and $\theta_T(e_b) = (n-k, k)$ and without loss of generality assume that $i \geq k$.  If $k=i$ then $\theta_T(\{e_a, e_b\})= re(n-2i, i,i)$  and if $k<i$ then $\theta_T(\{e_a, e_b\})$ is either $\textrm{re}(n-i-k, i, k)$ or $\textrm{re}(n-i, i-k, k)$.
\end{proposition}
\begin{proof}
Upon removing $e_a$, the tree $T$ is divided into two connected components containing $n-i$ and $i$ vertices.  If $i=k$, then we can only remove $e_b$ from the component with $n-i$ edges, hence the result follows.  If $k<i$,  $e_b$ may be in either connected component. Its removal leaves one of these components unchanged. Therefore, either $n-i$ or $i$ must be present in $\theta_T(\{e_a, e_b\})$. Similarly, by removing $e_b$ before $e_a$, we see that either $n-k$ or $k$ must be present in $\theta_T(\{e_a, e_b\})$. However, both $n-i$ and $n-k$ cannot be present in $\theta_T(\{e_a, e_b\})$ since $n-i + n-k \geq n$. Similarly, both $i$ and $n-k$ cannot be present in $\theta_T(\{e_a, e_b\})$, since $i + n-k \geq n$ (recalling the fact that $k < i$). Therefore, $\theta_T(\{e_a, e_b\})$ must either contain both $k$ and $i$ or it must contain $n-i$ and $k$. Since the sum of the numbers in the partition of $n$ must add to $n$, it follows that $\theta_T(\{e_a, e_b\})$ is either re$(n-i-k, i, k)$ or re$(n-i, i-k, k)$.
\end{proof}

\begin{definition} Let $T$ be a tree. We say that the distinct edges $e_a$ and $e_b$ of $T$ {\bf attract} if there is a path in $T$ containing both $e_a$ and $e_b$ and having a centroid of $T$ as one endpoint. Otherwise, the edges $e_a$ and $e_b$ {\bf repel}.
\end{definition}

This terminology is motivated by imagining each edge of $T$ as a bar magnet, the positive side of which is pointed towards the centroid(s). In the same way that two magnets attract if the positive side of one is pointed toward the negative side of the other, we say that two edges attract if a path connects the ``positive'' side of one edge to the ``negative'' side of the other. It can be checked that this corresponds to the formal definition of {\em attract} given above. On the other hand, two edges repel if a path connects the ``positive'' side of one to the ``positive'' side of the other. Figure \ref{fig:attractedges} shows an example: the dotted edges repel $e$, while the thick edges attract $e$. Vertex $c$ is the centroid of the tree. From the definition, it can be checked that if $T$ has two centroids, then the edge connecting them attracts every edge in $T$. We proceed by exploring a few properties of edge attraction and repulsion.

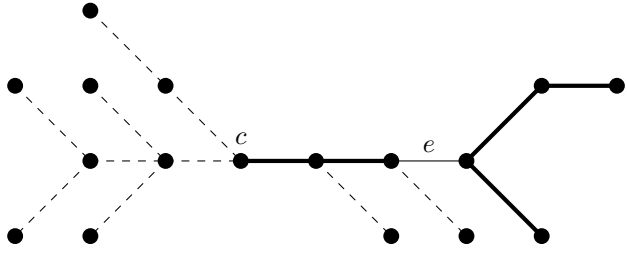
\begin{figure}[ht]
\centering
\begin{tikzpicture}

\tikzstyle{every node} = [node distance=1.0cm]

\draw (1,1) coordinate (A13);
\draw (2,1) coordinate (A14);
\draw (6,1) coordinate (A15);
\draw (7,1) coordinate (A16);
\draw (8,1) coordinate (A17);

\draw (2,2) coordinate (A7);
\draw (3,2) coordinate (A8);
\draw (4,2) coordinate (A9) node[above=3pt] {$c$};
\draw (5,2) coordinate (A10);
\draw (6,2) coordinate (A11);
\draw (7,2) coordinate (A12);

\draw (1,3) coordinate (A2);
\draw (2,3) coordinate (A3);
\draw (3,3) coordinate (A4);
\draw (8,3) coordinate (A5);
\draw (9,3) coordinate (A6);

\draw (2,4) coordinate (A1);

\draw[draw=black, dashed] (A1) -- (A4);
\draw[draw=black, dashed] (A2) -- (A7);
\draw[draw=black, dashed] (A3) -- (A8);
\draw[draw=black, dashed] (A4) -- (A9);
\draw[draw=black, ultra thick] (A5) -- (A12);
\draw[draw=black, ultra thick] (A5) -- (A6);

\draw[draw=black, dashed] (A7) -- (A8);
\draw[draw=black, dashed] (A8) -- (A9);
\draw[draw=black, ultra thick] (A9) -- (A10);
\draw[draw=black, ultra thick] (A10) -- (A11);
\draw[draw=black, thin] (A11) -- (A12)  node [midway, above] {$e$};

\draw[draw=black, dashed] (A7) -- (A13);
\draw[draw=black, dashed] (A8) -- (A14);
\draw[draw=black, dashed] (A10) -- (A15);
\draw[draw=black, dashed] (A11) -- (A16);
\draw[draw=black, ultra thick] (A12) -- (A17);

  \fill (A1) circle (3pt) 
  (A2) circle (3pt)
  (A3) circle (3pt)
  (A4) circle (3pt)
  (A5) circle (3pt)
  (A6) circle (3pt)
  (A7) circle (3pt)
  (A8) circle (3pt)
  (A9) circle (3pt)
  (A10) circle (3pt)
  (A11) circle (3pt)
  (A12) circle (3pt)
  (A13) circle (3pt)
  (A14) circle (3pt)
  (A15) circle (3pt)
  (A16) circle (3pt)
  (A17) circle (3pt);
  
\end{tikzpicture}
\caption{The bold edges attract $e$.}
\label{fig:attractedges}
\end{figure}

\begin{proposition} \label{theta_to_attract}
Let $e_a$ and $e_b$ be distinct edges of a tree $T$ with $\theta_T(e_a) = (n-i, i)$ and $\theta_T(e_b) = (n-k, k)$ with $i \geq k$. Then $\theta_T(e_a, e_b) = re(n-i, i-k, k)$ if and only if $e_a$ and $e_b$ attract.
\end{proposition}
\begin{proof}
Suppose we remove $e_a$ from $T$ to form a forest with two connected components: $T_{n-i}$ which contains $n-i$ vertices and $T_i$ which contains $i$ vertices. We split the proof into two cases: either $T_{n-i}$ and $T_{i}$ contain one centroid of $T$ each, or all the centroids are contained in $T_{n-i}$.

In the first case, the endpoints of $e_a$ must be the centroids of $T$. In this case, by Lemma \ref{centroid}  $\theta_T(e_a) = (\frac{n}{2}, \frac{n}{2})$ and $e_a$ attracts every edge of $T$. From Proposition \ref{splits}, $\theta_T(e_a, e_b)$ is either $\textrm{re}(n-\frac{n}{2}-k, \frac{n}{2}, k)$ or $\textrm{re}(n-\frac{n}{2}, \frac{n}{2}-k, k)$. However, these partitions are identical. In this first case, $e_a$ and $e_b$ attract if and only if $\theta_T(e_a, e_b) = (\frac{n}{2}, \frac{n}{2} - k, k)$.

In the second case, only $T_{n-i}$ contains centroids of $T$. If $e_b$ is also contained in $T_{n-i}$, then every path from $e_b$ to a centroid would also be contained in $T_{n-i}$ and would avoid $e_a$. Therefore, $e_a$ attracts $e_b$ if and only if $e_b$ is in $T_i$. By the proof of Proposition \ref{splits}, $e_b$ is in $T_i$ exactly when $\theta_T(e_a, e_b) = (n-i, i-k, k)$.
\end{proof}

Notice that from Proposition \ref{theta_to_attract} we have that $\theta_T(e_a,e_b)$ is not a three part partition when $n=k$.  Thus we have the following corollary. 

\begin{corollary}
If $\theta_T(e_a)=\theta_T(e_b)$, then $e_a$ and $e_b$ repel.  
\end{corollary}

\begin{remark}\label{nosep}
We say that an edge $e_a$ \emph{separates} an edge $e_b$ from a vertex $v$ if $e_a$ and $v$ are in different connected components of the graph after removing the edge $e_a$. If a tree $T$ has a single centroid and $e_a$ separates $e_b$ from the centroid, then $\theta_T(e_a) > \theta_T(e_b)$.
\end{remark}
This remark follows from the definition of $\theta_T$. We are now ready to prove the main theorem of this section.

\begin{theorem}\label{barmagnet}
Given the set of edges $\{e_1, e_2, \ldots, e_{n-1}\}$ of a tree $T$ with a single centroid and the values $\theta_{T}(e_i)$ and $\theta_{T}(e_i, e_k)$ for all distinct edges $e_i, e_k \in T$, then $T$ can be constructed from this data. In other words,  the values  $\theta_{T}(e_i)$ and $\theta_{T}(e_i, e_k)$ determine the tree. 

\end{theorem}

\begin{proof}
After reindexing the edges, we may assume that $\theta_T(e_i) \geq \theta_T(e_{i+1})$ for all $i < n$.  We will construct $T$ by adding edges, one by one, to a forest of $n$ isolated vertices. We will add the edges in the order of their index. During each step, we will use the values of $\theta_T$ given to determine the unique way to add the next edge. This process will create a sequence of forests $T_i$, where $T_i$ is the resulting forest after adding the first $i$ edges.  In the course of the construction, we will see that for $i \geq 1$, $T_i$ consists of $n-1-i$ isolated vertices and one tree with $i+1$ vertices.\footnote{This explains our choice to use the letter $T_i$ to denote the forests, not $F_i$.}

Notice that the first edge placed must have the centroid as an endpoint. Clearly, $T_1$ consists of $e_1$ together with $n-2$ isolated vertices. Now assume that we have placed the first $i$ edges, and that $T_i$ consists of a tree and $n - 1 - i$ isolated vertices. We wish to place the $(i+1)^{\textrm{st}}$ edge. We will show two things. First, $e_{i+1}$ must be attached to the tree with $i$ edges in $T_i$ (therefore $T_{i+1}$ will consist of a tree and $n - i$ isolated vertices). Second, we will show that there is one unique permissible way to attach the edge $e_{i+1}$ to the tree in $T_i$.

If we suppose that $e_{i+1}$ is not attached to the tree with $i$ edges in the forest  $T_i$, then in $T$ the edge $e_{i+1}$ will be separated from the centroid by an edge $e_j$ with $j > i+1$. By our indexing of the edges, $\theta_T(e_{i+1}) \geq \theta_{T}(e_j)$. However, this is forbidden by Remark \ref{nosep}. Therefore, $e_{i+1}$ is attached to the tree in $T_i$.

To see where $e_{i+1}$ must be attached, we need only look at which edges it attracts. By Proposition \ref{theta_to_attract}, we can determine which edges of $T_i$ attract $e_{i+1}$ and which edges of $T_i$ repel $e_{i+1}$ from the values $\theta(e_i, e_j)$ (for $e_j \in T_i$). The set of edges of $T_i$ that attract $e_{i+1}$ form a path from the centroid to the vertex $v$ that is adjacent to $e_{i+1}$.   This is evident from the definition of attraction and repulsion. Since there is a unique path from the centroid to the vertex $v$, there is a unique way to attach the edge $e_{i+1}$ to the forest $T_i$.   Thus we extend $T_i$ to $T_{i+1}$ by placing $e_{i+1}$ at the end of the path created by edges that attract it. 
\end{proof}

At first glance, the conditions outlined in Theorem \ref{barmagnet} seem very strict, and one might wonder whether they may be relaxed. For example, perhaps it is possible to construct $T$ knowing only $\theta_{T}(e)$ for all edges $e$ of $T$. This is not the case. Consider the following trees.

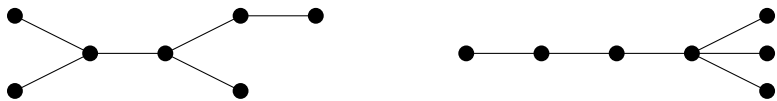
\begin{figure}[ht]
\centering
\begin{tikzpicture}

\tikzstyle{every node} = [node distance=1.0cm]

\draw (1,1.5) coordinate (A1) node[left=3pt] {};
\draw (4,1.5) coordinate (A2) node[left=3pt] {};
\draw (5,1.5) coordinate (A3) node[left=3pt] {};
\draw (2,1) coordinate (A4) node[left=3pt] {};
\draw (3,1) coordinate (A5) node[left=3pt] {};
\draw (1,.5) coordinate (A6) node[left=3pt] {};
\draw (4,.5) coordinate (A7) node[left=3pt] {};

\draw (11,1.5) coordinate (B1) node[left=3pt] {};
\draw (7,1) coordinate (B2) node[left=3pt] {};
\draw (8,1) coordinate (B3) node[left=3pt] {};
\draw (9,1) coordinate (B4) node[left=3pt] {};
\draw (10,1) coordinate (B5) node[left=3pt] {};
\draw (11,1) coordinate (B6) node[left=3pt] {};
\draw (11,.5) coordinate (B7) node[left=3pt] {};

\draw(A1) -- (A4);
\draw(A2) -- (A5);
\draw(A2) -- (A3);
\draw(A4) -- (A5);
\draw(A4) -- (A6);
\draw(A5) -- (A7);

\draw(B1) -- (B5);
\draw(B2) -- (B3);
\draw(B3) -- (B4);
\draw(B4) -- (B5);
\draw(B5) -- (B6);
\draw(B5) -- (B7);

  \fill (A1) circle (3pt) 
  (A2) circle (3pt)
  (A3) circle (3pt)
  (A4) circle (3pt)
  (A5) circle (3pt)
  (A6) circle (3pt)
  (A7) circle (3pt)

  (B1) circle (3pt) 
  (B2) circle (3pt)
  (B3) circle (3pt)
  (B4) circle (3pt)
  (B5) circle (3pt)
  (B6) circle (3pt)
  (B7) circle (3pt);
  
\end{tikzpicture}
\caption{Two trees with the same $\{\theta_T(e)\}_{e \in E(G)}$ data}
\label{fig:fivepointtwo}
\end{figure}
Notice that both graphs have exactly one edge that divides the vertices into a $(4,3)$ partition, one that divides the vertices into a $(5,2)$ partition, and four edges that divide the vertices in a $(6,1)$ partition. Therefore, knowing $\theta_{T}(e)$ for all the edges $e$ in a tree $T$ is insufficient information to construct a tree.

%
%
%
%
%
%
However, in the next theorem we show that we can reconstruct $T$ knowing only $\theta_T(e_i, e_j)$ for all pairs of edges $e_i, e_j$ in $T$ in the case that $T$ has only one centroid.   In the proof of the theorem we will refer to an edge adjacent to a vertex of degree 1 as a leaf-edge.   In the literature it is the vertex that is called a leaf, but to simplify the exposition we will refer to the edge as a leaf-edge. 

\begin{theorem}\label{barmagnet2} 
A tree $T$ with a single centroid is uniquely determined by the data $\theta_{T}(e_i, e_k)$ for all distinct edges $e_i, e_k \in T$.
\end{theorem}
\begin{proof}
The number of vertices in the tree is determined by the data $\{\theta_T(e_i, e_j)\}$ by counting the number of edges and adding one. Because there is only one single-centroid tree on $n$ vertices for $n \leq 4$, we may assume that $n > 4$.  We will prove the statement by first showing that we can determine precisely which edges of $T$ must be leaf-edges, then using this information to determine $\theta_{T}(e)$ for all $e \in E(T)$. Invoking Theorem \ref{barmagnet} will complete the proof.

Let $L \subseteq E(T)$ consist of all edges $e_i \in E(T)$ such that 
\[
\# \{e_j : e_j \neq e_i, \theta(\{e_i, e_j\}) = (n-2, 1, 1)\} \geq 2.
\]
We claim that $L$ is the set of leaf-edges of $T$.  To show that every leaf-edge is in $L$, we study two cases:  $T$ has three or more leaves or $T$ has two leaves.  If $T$ has three or more leaves then for any leaf-edge $e_i$ there are at least two ways to remove another leaf-edge and get the partition $(n-2,1,1)$.  If $T$ has only two leaves, then $T$ is a path on $n>4$ vertices. In this case, for any leaf-edge $e_i$ we get the partition $(n-2,1,1)$ by removing $e_i$ and the edge incident to it, or by removing $e_i$ and the other leaf-edge.  Hence, any leaf-edge is contained in $L$.   Suppose now that an edge $e_i$ is not a leaf-edge, we will show that $e_i \notin L$.  Indeed, if $e_i$ is not a leaf-edge, then $\theta(\{e_i, e_j\}) = (n-2, 1, 1)$ only if $e_j$ is a leaf-edge that meets $e_i$ at a vertex of degree 2. Therefore, unless $e_i$ is adjacent to two leaf-edges through vertices of degree 2, it is not in $L$. The only case where this phenomenon could occur is in the path with four vertices, contradicting $n > 4$. Therefore, the set $L$ consists precisely of the leaves of $T$. 

Next, consider any $e_k \notin L$. We can completely determine $\theta_T(e_k)$ by studying how the vertices of $T$ are partitioned when $e_k$ and a leaf-edge are removed. That is, it is clear that there is at least one leaf-edge in $T$ that repels $e_k$, and at least one leaf-edge in $T$ that attracts $e_k$. Because every leaf-edge in $T$ either attracts $e_k$ or repels $e_k$, it follows from Proposition $\ref{theta_to_attract}$ that $\{\theta_T(e_k, e_j): e_j \in L\} =  \{\textrm{re}(n-i-1, i, 1), \textrm{re}(n-i, i-1, 1)\}$ for some $i$. By letting $a = \textrm{max}\{n - i, i\}$, we see that $\theta_T(e_k) = (a, n-a)$.

We have determined $\theta_T(e_k)$ for all $e_k \in E(T)$. Using Theorem \ref{barmagnet} completes the proof.
\end{proof}

To illustrate Theorem \ref{barmagnet2}, suppose we are given the following doubleton sets of $\theta_T$-images of edges of a tree with a single centroid, and we are asked to construct the tree.

\tiny
\begin{center}
\begin{tabular}{@{}c@{}||@{\ }c@{\ }|@{\ }c@{\ }|@{\ }c@{\ }|@{\ }c@{\ }|@{\ }c@{\ }|@{\ }c@{\ }|@{\ }c@{\ }|@{\ }c@{\ }|@{\ }c@{\ }|@{\ }c@{\ }|@{\ }c@{\ }}
 & ${e_1}$ & ${e_2}$ & ${e_3}$ & ${e_4}$ & ${e_5}$ & ${e_6}$ & ${e_7}$ & ${e_8}$ & ${e_9}$ & ${e_{10}}$ & ${e_{11}}$ \\ \hline \hline

${e_1}$ & & & & & & & & & & &  \\ \hline
${e_2}$ & $(10, 2, 1)$ & & & & & & & & & & \\ \hline
${e_3}$ & $(7, 5, 1)$ & $(7, 4, 2)$ & & & & & & & & & \\ \hline
${e_4}$ & $(10, 2, 1)$ & $(9, 2, 2)$ & $(6, 5, 2)$ & & & & & & & & \\ \hline
${e_5}$ & $(11, 1, 1)$ & $(10, 2, 1)$ & $(6, 6, 1)$ & $(10, 2, 1)$ & & & & & & & \\ \hline
${e_6}$ & $(11, 1, 1)$ & $(10, 2, 1)$ & $(6, 6, 1)$ & $(11, 1, 1)$ & $(11, 1, 1)$ & & & & & & \\ \hline
${e_7}$ & $(11, 1, 1)$ & $(11, 1, 1)$ & $(7, 5, 1)$ & $(10, 2, 1)$ & $(11, 1, 1)$ & $(11, 1, 1)$ & & & & & \\ \hline
${e_8}$ & $(11, 1, 1)$ & $(10, 2, 1)$ & $(6, 6, 1)$ & $(10, 2, 1)$ & $(11, 1, 1)$ & $(11, 1, 1)$ & $(11, 1, 1)$ & & & & \\ \hline
${e_9}$ & $(9, 3, 1)$ & $(8, 3, 2)$ & $(6, 4, 3)$ & $(8, 3, 2)$ & $(10, 2, 1)$ & $(9, 3, 1)$ & $(9, 3, 1)$ & $(10, 2, 1)$ & & & \\ \hline
${e_{10}}$ & $(9, 3, 1)$ & $(10, 2, 1)$ & $(7, 3, 3)$ & $(8, 3, 2)$ & $(9, 3, 1)$ & $(9, 3, 1)$ & $(10, 2, 1)$ & $(9, 3, 1)$ & $(7, 3, 3)$ & & \\ \hline
${e_{11}}$ & $(11, 1, 1)$ & $(10, 2, 1)$ & $(7, 5, 1)$ & $(10, 2, 1)$ & $(11, 1, 1)$ & $(11, 1, 1)$ & $(11, 1, 1)$ & $(11, 1, 1)$ & $(9, 3, 1)$ & $(9, 3, 1)$ & \\ \hline
${e_{12}}$ & $(6, 6, 1)$ & $(6, 5, 2)$ & $(6, 6, 1)$ & $(7, 4, 2)$ & $(7, 5, 1)$ & $(7, 5, 1)$ & $(6, 6, 1)$ & $(7, 5, 1)$ & $(7, 3, 3)$ & $(6, 4, 3)$ & $(6, 6, 1)$ 
\end{tabular}
\end{center}
\normalsize

\begin{center}
Table 5.4: $\theta_T$-images of two-element sets
\end{center}

The first step is to identify the leaves. Following the proof of Theorem \ref{barmagnet2}, an edge $e$ is a leaf if and only if there are at least two other edges $e_i$ and $e_j$ for which $\theta_T(e, e_i) = \theta_T(e, e_j) = (11, 1, 1)$.  The leaves of our tree are the edges $L = \{e_1, e_5, e_6, e_7, e_8, e_{11}\}$. Therefore, $\theta_T(e) = (12, 1)$ for these edges.

The next step is to determine the $\theta_T$-images of all other edges in the tree. For an edge $e \notin L$, we check the set of partitions $\{\theta_T(e, e_l): e_l \in L\}$. For example, this set of partitions for $e_{10}$ is $\{(9, 3, 1), (10, 2, 1)\}$. Since $10$ is the maximum number in any of these partitions, $\theta_T(e_{10}) = (10, 3)$. We can use similar reasoning to deduce the $\theta_T$-image of every edge.

\tiny
\begin{center}
\begin{tabular}{@{}c@{}||@{\ }c@{\ }|@{\ }c@{\ }|@{\ }c@{\ }|@{\ }c@{\ }|@{\ }c@{\ }|@{\ }c@{\ }|@{\ }c@{\ }|@{\ }c@{\ }|@{\ }c@{\ }|@{\ }c@{\ }|@{\ }c@{\ }|@{\ }c@{\ }}
 Edge & ${e_1}$ & ${e_2}$ & ${e_3}$ & ${e_4}$ & ${e_5}$ & ${e_6}$ & ${e_7}$ & ${e_8}$ & ${e_9}$ & ${e_{10}}$ & ${e_{11}}$ & ${e_{12}}$ \\ \hline
$\theta_T$-image & (12, 1) & (11, 2) & (7, 6) & (11, 2) & (12, 1) & (12, 1) & (12, 1) & (12, 1) & (10, 3) & (10, 3) & (12, 1) & (7, 6)
\end{tabular}
\end{center}
\normalsize
\begin{center}
Table 5.4: $\theta$-images of the edges in Table 5.4
\end{center}

The next step is to index the edges $\{e_i^\prime\}_{i = 1}^{n-1}$ so that $\theta_T(e_i^\prime) \geq \theta_T(e_{i+1}^\prime)$. We will use the indexing 
$$e_1^\prime, e_2^\prime, \dots e_{12}^\prime = e_3, e_{12}, e_9, e_{10}, e_2, e_4, e_1, e_5, e_6, e_7, e_8, e_{11}\textrm{.}$$
This is the order in which the edges will be added. Notice that we have some freedom in our choice of edge ordering. For example, we may swap the position of $e_9$ and $e_{10}$ without changing the result of this example. The forests $T_1, T_2, \dots, T_{12}$ are shown below (their isolated vertices are omitted for clarity). To create $T_{i+1}$ from $T_i$, we add the edge $e_{i+1}^\prime$ by checking which edges attract it. The attracting edges are in bold, and both the centroid $c$ and the newly added edge $e_{i+1}^\prime$ are labelled. 

\begin{center}
\begin{tikzpicture}

\draw (0, 2) coordinate (A0) node[below=3pt] {$T_1 =$};

\draw (1, 1.5) coordinate (A1) node[below=3pt] {$c$};
\draw (1.5, 1.5) coordinate (A2);
\footnotesize
\draw(A1) -- (A2)  node [midway, above = 1pt] {$e_3$};;
\normalsize
  \fill (A1) circle (2pt) 
  (A2) circle (2pt);

\draw (5, 2) coordinate (B0) node[below=3pt] {$T_2 =$};

\draw (6, 1.5) coordinate (B1);
\draw (6.5, 1.5) coordinate (B2) node[below=3pt] {$c$};
\draw (7.0, 1.5) coordinate (B3);
\footnotesize
\draw (B1) -- (B2)  node [midway, above = 1pt] {$e_{12}$};
\normalsize
\draw (B3) -- (B2);

  \fill (B1) circle (2pt) 
  (B2) circle (2pt)
  (B3) circle (2pt);

\draw (10, 2) coordinate (C0) node[below=3pt] {$T_3 =$};

\draw (11, 1.0) coordinate (C1);
\draw (11.5, 1.5) coordinate (C2);
\draw (12, 1.5)   coordinate (C3) node[below=3pt] {$c$};
\draw (12.5, 1.5) coordinate (C4);

\draw (14.5, 1.5) coordinate (END);
\footnotesize
\draw (C1) -- (C2)  node [midway] {$\ \ \ \ e_9$};
\normalsize
\draw[draw=black, ultra thick] (C2) -- (C3);
\draw (C4) -- (C3);

  \fill (C1) circle (2pt) 
  (C2) circle (2pt)
  (C3) circle (2pt)
  (C4) circle (2pt);

\end{tikzpicture}
\end{center}

\begin{center}
\begin{tikzpicture}

\draw (0, 2) coordinate (A0) node[below=3pt] {$T_4 =$};

\draw (1, 1) coordinate (A1);
\draw (1.5, 1.5) coordinate (A2);
\draw (2, 1.5) coordinate (A3) node[below=3pt] {$c$};
\draw (2.5, 1.5) coordinate (A4);
\draw (3, 1) coordinate (A5);

\draw(A1) -- (A2);
\draw(A2) -- (A3);
\draw[draw=black, ultra thick](A3) -- (A4);
\footnotesize
\draw(A4) -- (A5)  node [midway] {$\ \ \ \ \ e_{10}$};
\normalsize

  \fill (A1) circle (2pt) 
  (A2) circle (2pt)
  (A3) circle (2pt)
  (A4) circle (2pt)
  (A5) circle (2pt);

\draw (5, 2) coordinate (B0) node[below=3pt] {$T_5 =$};

\draw (6, 1.0) coordinate (B1);
\draw (6.5, 1.5) coordinate (B2);
\draw (7.0, 1.5) coordinate (B3) node[below=3pt] {$c$};
\draw (7.5, 1.5) coordinate (B4);
\draw (8.0, 1.0) coordinate (B5);
\draw (8.5, 1.0) coordinate (B6);

\draw (B1) -- (B2);
\draw (B2) -- (B3);
\draw[draw=black, ultra thick] (B3) -- (B4);
\draw[draw=black, ultra thick] (B4) -- (B5);
\footnotesize
\draw (B5) -- (B6)  node [midway, above = 1pt] {$e_2$};
\normalsize

  \fill (B1) circle (2pt) 
  (B2) circle (2pt)
  (B4) circle (2pt)
  (B5) circle (2pt)
	(B6) circle (2pt)
  (B3) circle (2pt);

\draw (10, 2) coordinate (C0) node[below=3pt] {$T_6 =$};

\draw (11.0, 1.0) coordinate (C1);
\draw (11.0, 2.0) coordinate (C2);
\draw (11.5, 1.5)   coordinate (C3);
\draw (12.0, 1.5) coordinate (C4) node[below=3pt] {$c$};
\draw (12.5, 1.5) coordinate (C5);
\draw (13.0, 1.0) coordinate (C6);
\draw (13.5, 1.0) coordinate (C7);

\draw (14.5, 1.5) coordinate (END);

\draw (C1) -- (C3);
\footnotesize
\draw (C2) -- (C3)  node [midway] {$\ \ \ e_4$};
\normalsize
\draw[draw=black, ultra thick] (C3) -- (C4);
\draw (C4) -- (C5);
\draw (C5) -- (C6);
\draw (C6) -- (C7);

  \fill (C1) circle (2pt) 
  (C2) circle (2pt)
  (C3) circle (2pt)
  (C4) circle (2pt)
  (C5) circle (2pt)
  (C6) circle (2pt)
  (C7) circle (2pt);

\end{tikzpicture}
\end{center}

\begin{center}
\begin{tikzpicture}

\draw (0, 2) coordinate (A0) node[below=3pt] {$T_7 =$};

\draw (1, 1) coordinate (A1);
\draw (1, 2) coordinate (A2);
\draw (1.5, 1.5) coordinate (A3);
\draw (2, 1.5) coordinate (A4) node[below=3pt] {$c$};
\draw (2.5, 1.5) coordinate (A5);
\draw (3, 1) coordinate (A6);
\draw (3, 1.5) coordinate (A7);
\draw (3.5, 1) coordinate (A8);

\draw(A1) -- (A3);
\draw(A2) -- (A3);
\draw(A3) -- (A4);
\draw[draw=black, ultra thick] (A4) -- (A5);
\draw(A5) -- (A6);
\footnotesize
\draw(A5) -- (A7)  node [midway, above = 1pt] {$e_1$};
\normalsize
\draw(A6) -- (A8);

  \fill (A1) circle (2pt) 
  (A2) circle (2pt)
  (A3) circle (2pt)
  (A4) circle (2pt)
  (A5) circle (2pt)
  (A6) circle (2pt)
  (A7) circle (2pt)
  (A8) circle (2pt);

\draw (5, 2) coordinate (B0) node[below=3pt] {$T_8 =$};

\draw (6, 1.0) coordinate (B1);
\draw (6.5, 1) coordinate (B2);
\draw (6.5, 2) coordinate (B3);
\draw (7.0, 1.5) coordinate (B4);
\draw (7.5, 1.5) coordinate (B5) node[below=3pt] {$c$};
\draw (8.0, 1.5) coordinate (B6);
\draw (8.5, 1.0) coordinate (B7);
\draw (8.5, 1.5) coordinate (B8);
\draw (9.0, 1.0) coordinate (B9);

\footnotesize
\draw (B1) -- (B2)  node [midway, above = 1pt] {$e_5$};
\normalsize
\draw[draw=black, ultra thick] (B2) -- (B4);
\draw (B3) -- (B4);
\draw[draw=black, ultra thick] (B4) -- (B5);
\draw (B5) -- (B6);
\draw (B7) -- (B6);
\draw (B8) -- (B6);
\draw (B7) -- (B9);

  \fill (B1) circle (2pt) 
  (B2) circle (2pt)
  (B4) circle (2pt)
  (B5) circle (2pt)
	(B6) circle (2pt)
  (B3) circle (2pt)
  (B7) circle (2pt)
  (B8) circle (2pt)
  (B9) circle (2pt);

\draw (10, 2) coordinate (C0) node[below=3pt] {$T_9 =$};

\draw (11.0, 1.0) coordinate (C1);
\draw (11.0, 2.0) coordinate (C2);
\draw (11.5, 1.0)   coordinate (C3);
\draw (11.5, 2.0) coordinate (C4);
\draw (12.0, 1.5) coordinate (C5);
\draw (12.5, 1.5) coordinate (C6) node[below=3pt] {$c$};
\draw (13.0, 1.5) coordinate (C7);
\draw (13.5, 1.0) coordinate (C8);
\draw (13.5, 1.5) coordinate (C9);
\draw (14.0, 1.0) coordinate (C10);

\draw (14.5, 1.5) coordinate (END);

\draw (C1) -- (C3);
\footnotesize
\draw (C2) -- (C4)  node [midway, above = 1pt] {$e_6$};
\normalsize
\draw (C3) -- (C5);
\draw[draw=black, ultra thick] (C4) -- (C5);
\draw[draw=black, ultra thick] (C5) -- (C6);
\draw (C6) -- (C7);
\draw (C7) -- (C8);
\draw (C7) -- (C9);
\draw (C8) -- (C10);

  \fill (C1) circle (2pt) 
  (C2) circle (2pt)
  (C3) circle (2pt)
  (C4) circle (2pt)
  (C5) circle (2pt)
  (C6) circle (2pt)
  (C7) circle (2pt)
  (C8) circle (2pt)
  (C9) circle (2pt)
  (C10) circle (2pt);

\end{tikzpicture}
\end{center}

\begin{center}
\begin{tikzpicture}[style=very thin]

\draw (0, 2) coordinate (A0) node[below=3pt] {$T_{10} =$};

\draw (1, 1) coordinate (A1);
\draw (1, 2) coordinate (A2);
\draw (1.5, 1) coordinate (A3);
\draw (1.5, 2) coordinate (A4);
\draw (2.0, 1.5) coordinate (A5);
\draw (2.5, 1.5) coordinate (A6) node[below=3pt] {$c$};
\draw (3, 1.5) coordinate (A7);
\draw (3.5, 1) coordinate (A8);
\draw (3.5, 1.5) coordinate (A9);
\draw (4, 1) coordinate (A10);
\draw (4.5, 1) coordinate (A11);

\draw(A1) -- (A3);
\draw(A2) -- (A4);
\draw(A3) -- (A5);
\draw(A4) -- (A5);
\draw(A5) -- (A6);
\draw[draw=black, ultra thick] (A6) -- (A7);
\draw[draw=black, ultra thick] (A7) -- (A8);
\draw(A7) -- (A9);
\draw[draw=black, ultra thick] (A8) -- (A10);
\footnotesize
\draw(A10) -- (A11)  node [midway, above = 1pt] {$e_7$};
\normalsize

  \fill (A1) circle (2pt) 
  (A2) circle (2pt)
  (A3) circle (2pt)
  (A4) circle (2pt)
  (A5) circle (2pt)
  (A6) circle (2pt)
  (A7) circle (2pt)
  (A8) circle (2pt)
  (A9) circle (2pt)
  (A10) circle (2pt)
  (A11) circle (2pt);

\draw (5, 2) coordinate (B0) node[below=3pt] {$T_{11} =$};

\draw (6, 1.0) coordinate (B1);
\draw (6, 1.5) coordinate (B2);
\draw (6, 2) coordinate (B3);
\draw (6.5, 1) coordinate (B4);
\draw (6.5, 2) coordinate (B5);
\draw (7, 1.5) coordinate (B6);
\draw (7.5, 1.5) coordinate (B7) node[below=3pt] {$c$};
\draw (8.0, 1.5) coordinate (B8);
\draw (8.5, 1) coordinate (B9);
\draw (8.5, 1.5) coordinate (B10);
\draw (9, 1) coordinate (B11);
\draw (9.5, 1) coordinate (B12);

\draw (B1) -- (B4);
\footnotesize
\draw (B2) -- (B4)  node [midway] {$\ \ \ e_8$};
\normalsize
\draw (B3) -- (B5);
\draw[draw=black, ultra thick] (B4) -- (B6);
\draw (B5) -- (B6);
\draw[draw=black, ultra thick] (B7) -- (B6);
\draw (B8) -- (B7);
\draw (B8) -- (B9);
\draw (B8) -- (B10);
\draw (B9) -- (B11);
\draw (B11) -- (B12);

  \fill (B1) circle (2pt) 
  (B2) circle (2pt)
  (B4) circle (2pt)
  (B5) circle (2pt)
	(B6) circle (2pt)
  (B3) circle (2pt)
  (B7) circle (2pt)
  (B8) circle (2pt)
  (B9) circle (2pt)
  (B10) circle (2pt)
  (B11) circle (2pt)
  (B12) circle (2pt);

\draw (10, 2) coordinate (C0) node[below=3pt] {$T_{12} =$};

\draw (11.0, 1.0) coordinate (C1);
\draw (11.0, 1.5) coordinate (C2);
\draw (11.0, 2.0)   coordinate (C3);
\draw (11.5, 1.0) coordinate (C4);
\draw (11.5, 2.0) coordinate (C5);
\draw (12.0, 1.5) coordinate (C6);
\draw (12.5, 1.5) coordinate (C7) node[below=3pt] {$c$};
\draw (13.0, 1.5) coordinate (C8);
\draw (13.5, 1.0) coordinate (C9);
\draw (13.5, 1.5) coordinate (C10);
\draw (13.5, 2.0) coordinate (C11);
\draw (14.0, 1.0) coordinate (C12);
\draw (14.5, 1.0) coordinate (C13);

\draw (14.5, 1.5) coordinate (END);

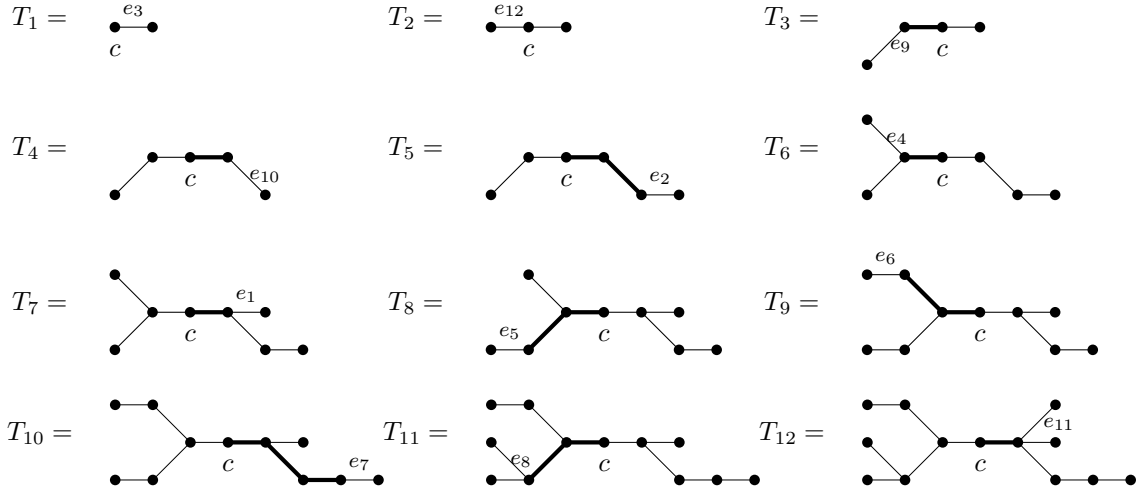
\draw (6.25, 0) coordinate (TITLE) node[below=3pt] {Figure 5.5: Construction of the tree satisfying the $\theta_T$-images in Table 5.4};

\draw (C1) -- (C4);
\draw (C2) -- (C4);
\draw (C3) -- (C5);
\draw (C4) -- (C6);
\draw (C5) -- (C6);
\draw (C6) -- (C7);
\draw[draw=black, ultra thick] (C7) -- (C8);
\draw (C8) -- (C9);
\footnotesize
\draw (C8) -- (C11)  node [midway] {$\ \ \ \ \ \ e_{11}$};
\normalsize
\draw (C8) -- (C10);
\draw (C9) -- (C12);
\draw (C12) -- (C13);

  \fill (C1) circle (2pt) 
  (C2) circle (2pt)
  (C3) circle (2pt)
  (C4) circle (2pt)
  (C5) circle (2pt)
  (C6) circle (2pt)
  (C7) circle (2pt)
  (C8) circle (2pt)
  (C9) circle (2pt)
  (C10) circle (2pt)
  (C11) circle (2pt)
  (C12) circle (2pt)
  (C13) circle (2pt);

\end{tikzpicture}
\end{center}

The forest $T_{12}$ is the unique tree with a single centroid satisfying the data in Table 5.4.  Theorem \ref{barmagnet2} shows that all single-centroid trees are classified by the $\theta$ images of 2-element sets of edges. One might ask whether we may lift the restriction in Theorem \ref{barmagnet2} that the tree must have a single cetroid. Unfortunately, we cannot. Figure 5.6 shows an example of two labelled graphs that have the same $\theta$ images of $2$-element sets of edges. A table of all the 2-element sets is included for reference.

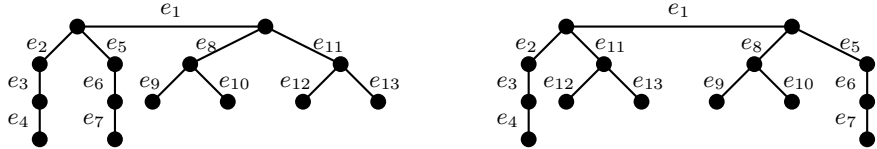
\begin{center}
\begin{tikzpicture}[style=thick] \label{two_centroid}

\draw (2,.75) coordinate (A13) node[left=3pt] {};
\draw (3,.75) coordinate (A14) node[left=3pt] {};
\draw (8.5,.75) coordinate (B13) node[left=3pt] {};
\draw (13,.75) coordinate (B14) node[left=3pt] {};

\draw (2,1.25) coordinate (A7) node[left=3pt] {};
\draw (3,1.25) coordinate (A8) node[left=3pt] {};
\draw (3.5,1.25) coordinate (A9) node[left=3pt] {};
\draw (4.5,1.25) coordinate (A10) node[left=3pt] {};
\draw (5.5,1.25) coordinate (A11) node[left=3pt] {};
\draw (6.5,1.25) coordinate (A12) node[left=3pt] {};
\draw (8.5,1.25) coordinate (B7) node[left=3pt] {};
\draw (9,1.25) coordinate (B8) node[left=3pt] {};
\draw (10,1.25) coordinate (B9) node[left=3pt] {};
\draw (11,1.25) coordinate (B10) node[left=3pt] {};
\draw (12,1.25) coordinate (B11) node[left=3pt] {};
\draw (13,1.25) coordinate (B12) node[left=3pt] {};

\draw (2,1.75) coordinate (A3) node[left=3pt] {};
\draw (3,1.75) coordinate (A4) node[left=3pt] {};
\draw (4,1.75) coordinate (A5) node[left=3pt] {};
\draw (6,1.75) coordinate (A6) node[left=3pt] {};
\draw (8.5,1.75) coordinate (B3) node[left=3pt] {};

\draw (9.5,1.75) coordinate (B4) node[left=3pt] {};
\draw (11.5,1.75) coordinate (B5) node[left=3pt] {};
\draw (13,1.75) coordinate (B6) node[left=3pt] {};

\draw (2.5,2.25) coordinate (A1) node[left=3pt] {};
\draw (5,2.25) coordinate (A2) node[left=3pt] {};
\draw (9,2.25) coordinate (B1) node[left=3pt] {};
\draw (12,2.25) coordinate (B2) node[left=3pt] {};

\draw(A1) -- (A2) node [midway, above] {\small$e_1$};
\draw(A1) -- (A4) node [midway, right] {\small$e_5$};
\draw(A1) -- (A3) node [midway, left] {\small$e_2$};
\draw(A2) -- (A5) node [midway, left] {\small$e_8$};
\draw(A2) -- (A6) node [midway, right] {\small$e_{11}$};
\draw(A3) -- (A7) node [midway, left] {\small$e_3$};
\draw(A4) -- (A8) node [midway, left] {\small$e_6$};
\draw(A5) -- (A9) node [midway, left] {\small$e_9$};
\draw(A5) -- (A10) node [midway, right] {\small$e_{10}$};
\draw(A6) -- (A11) node [midway, left] {\small$e_{12}$};
\draw(A6) -- (A12) node [midway, right] {\small$e_{13}$};
\draw(A7) -- (A13) node [midway, left] {\small$e_4$};
\draw(A8) -- (A14) node [midway, left] {\small$e_7$};

\draw(B1) -- (B2) node [midway, above] {\small$e_1$};
\draw(B1) -- (B4) node [midway, right] {\small$e_{11}$};
\draw(B1) -- (B3) node [midway, left] {\small$e_2$};
\draw(B2) -- (B5) node [midway, left] {\small$e_8$};
\draw(B2) -- (B6) node [midway, right] {\small$e_5$};
\draw(B3) -- (B7) node [midway, left] {\small$e_3$};
\draw(B4) -- (B8) node [midway, left] {\small$e_{12}$};
\draw(B4) -- (B9) node [midway, right] {\small$e_{13}$};
\draw(B5) -- (B10) node [midway, left] {\small$e_9$};
\draw(B5) -- (B11) node [midway, right] {\small$e_{10}$};
\draw(B6) -- (B12) node [midway, left] {\small$e_6$};
\draw(B7) -- (B13) node [midway, left] {\small$e_4$};
\draw(B12) -- (B14) node [midway, left] {\small$e_7$};

  \fill (A1) circle (3pt) 
  (A2) circle (3pt)
  (A3) circle (3pt)
  (A4) circle (3pt)
  (A5) circle (3pt)
  (A6) circle (3pt)
  (A7) circle (3pt)
  (A8) circle (3pt)
  (A9) circle (3pt)
  (A10) circle (3pt)
  (A11) circle (3pt)
  (A12) circle (3pt)
  (A13) circle (3pt)
  (A14) circle (3pt)
  (B1) circle (3pt) 
  (B2) circle (3pt)
  (B3) circle (3pt)
  (B4) circle (3pt)
  (B5) circle (3pt)
  (B6) circle (3pt)
  (B7) circle (3pt)
  (B8) circle (3pt)
  (B9) circle (3pt)
  (B10) circle (3pt)
  (B11) circle (3pt)
  (B12) circle (3pt)
  (B13) circle (3pt)
  (B14) circle (3pt);
\end{tikzpicture}
\end{center}

\tiny
\begin{center}
\begin{tabular}{@{}c@{}||@{\ }c@{\ }|@{\ }c@{\ }|@{\ }c@{\ }|@{\ }c@{\ }|@{\ }c@{\ }|@{\ }c@{\ }|@{\ }c@{\ }|@{\ }c@{\ }|@{\ }c@{\ }|@{\ }c@{\ }|@{\ }c@{\ }|@{\ }c@{\ }}
 & $e_1$ & $e_2$ & $e_3$ & $e_4$ & $e_5$ & $e_6$ & $e_7$ & $e_8$ & $e_9$ & $e_{10}$ & $e_{11}$ & $e_{12}$ \\ \hline \hline

$e_1$ & & & & & & & & & & & &  \\ \hline
$e_2$ & $(7, 4, 3)$ & & & & & & & & & & & \\ \hline
$e_3$ & $(7, 5, 2)$ & $(11, 2, 1)$ & & & & & & & & & & \\ \hline
$e_4$ & $(7, 6, 1)$ & $(11, 2, 1)$ & $(12, 1, 1)$ & & & & & & & & & \\ \hline
$e_5$ & $(7, 4, 3)$ & $(8, 3, 3)$ & $(9, 3, 2)$ & $(10, 3, 1)$ & & & & & & & & \\ \hline
$e_6$ & $(7, 5, 2)$ & $(9, 3, 2)$ & $(10, 2, 2)$ & $(11, 2, 1)$ & $(11, 2, 1)$ & & & & & & & \\ \hline
$e_7$ & $(7, 6, 1)$ & $(10, 3, 1)$ & $(11, 2, 1)$ & $(12, 1, 1)$ & $(11, 2, 1)$ & $(12, 1, 1)$ & & & & & & \\ \hline
$e_8$ & $(7, 4, 3)$ & $(8, 3, 3)$ & $(9, 3, 2)$ & $(10, 3, 1)$ & $(8, 3, 3)$ & $(9, 3, 2)$ & $(10, 3, 1)$ & & & & & \\ \hline
$e_9$ & $(7, 6, 1)$ & $(10, 3, 1)$ & $(11, 2, 1)$ & $(12, 1, 1)$ & $(10, 3, 1)$ & $(11, 2, 1)$ & $(12, 1, 1)$ & $(11, 2, 1)$ & & & & \\ \hline
$e_{10}$ & $(7, 6, 1)$ & $(10, 3, 1)$ & $(11, 2, 1)$ & $(12, 1, 1)$ & $(10, 3, 1)$ & $(11, 2, 1)$ & $(12, 1, 1)$ & $(11, 2, 1)$ & $(12, 1, 1)$ & & & \\ \hline
$e_{11}$ & $(7, 4, 3)$ & $(8, 3, 3)$ & $(9, 3, 2)$ & $(10, 3, 1)$ & $(8, 3, 3)$ & $(9, 3, 2)$ & $(10, 3, 1)$ & $(8, 3, 3)$ & $(10, 3, 1)$ & $(10, 3, 1)$ & & \\ \hline
$e_{12}$ & $(7, 6, 1)$ & $(10, 3, 1)$ & $(11, 2, 1)$ & $(12, 1, 1)$ & $(10, 3, 1)$ & $(11, 2, 1)$ & $(12, 1, 1)$ & $(10, 3, 1)$ & $(12, 1, 1)$ & $(12, 1, 1)$ & $(11, 2, 1)$ & \\ \hline
$e_{13}$ & $(7, 6, 1)$ & $(10, 3, 1)$ & $(11, 2, 1)$ & $(12, 1, 1)$ & $(10, 3, 1)$ & $(11, 2, 1)$ & $(12, 1, 1)$ & $(10, 3, 1)$ & $(12, 1, 1)$ & $(12, 1, 1)$ & $(11, 2, 1)$ & $(12, 1, 1)$
\label{Table 1}
\end{tabular}
\end{center}
\normalsize
\begin{center}
Figure 5.6: Two trees with the same $\theta_T$ images of 2-element sets
\end{center}

The data of the $\theta_T$-images of sets of edges of a tree are related to the data of the coefficients of its chromatic symmetric function. To understand this relationship, recall that when we write the chromatic symmetric function of any graph using the basis of power symmetric functions, each coefficient encapsulates information about subgraphs of a certain vertex partition type. For a graph with cycles, it is possible that two subgraphs have the same vertex partition type, but a different parity in their number of edges, so that their contributions to the sum in Equation (\ref{stan}) cancel one another out. This does not occur in forests.

\begin{proposition}\label{symtre}
For forests $F$ and $H$, ${\bf X}_F = {\bf X}_H$ if and only if 
$$\#\{S \subseteq E(F): \pi(S) = \lambda\} = \#\{S \subseteq E(H): \pi(S) = \lambda\}\textrm{,}$$
for all partitions $\lambda$ of $\#V(F)$.
\end{proposition}
\begin{proof}
All induced subgraphs of a forest with a given vertex partition have the same number of connected components, and therefore the same number of edges. Let $k$ be the number of parts in a partiton $\lambda$ of $\#V(F)$. Note that $\#S = \#V(F) - k$. Then the coefficient of $p_\lambda$ in ${\bf X}_F$ is equal to
$$(-1)^{\#V(F)-k}\#\{S \subseteq E(F): \pi(S) = \lambda\}\textrm{.}$$
The proposition follows immediately.
\end{proof}

Notice that $\pi(S) = \theta_F(E(F) - S)$, so a corollary of Proposition \ref{symtre} is the following.

\begin{corollary}\label{cor:lastcor}
For forests $F$ and $H$, ${\bf X}_F = {\bf X}_H$ if and only if 
$$\#\{S \subseteq E(F): \theta_F(S) = \lambda\} = \#\{S \subseteq E(H): \theta_F(S) = \lambda\}\textrm{,}$$
for all partitions $\lambda$ of $\#V(F)$.
\end{corollary}

It is very tempting to misinterpret Theorem \ref{barmagnet2} and claim that together with Corollary \ref{cor:lastcor} it proves that ${\bf X}_T$ determines trees with single centroids. After all, it seems that the $\theta_T$ images of 2-element sets of $E(T)$ are determined from ${\bf X}_T$ by the above corollary. However, it is important to remember that only the {\em number} (and not the labels) of pairs of edges giving a certain vertex partition under $\theta_T$ is determined from ${\bf X}_T$. To illustrate this important distinction, consider the trees in Figure \ref{one_centroid} and the chart of their $\theta_T$-images of 2-element sets. Each 3-part partition appears the same number of times in each chart, but they are arranged differently within the chart.  The trees in Figure \ref{one_centroid} are not isomorphic and it can be checked that their chromatic symmetric functions are not equal: the entire function is too long to print here, but when written in the power-sum basis the tree on the left has $-9p_{(8,5,1,1)}$ as a summand while the tree on the right has $-8p_{(8,5,1,1)}$ as a summand.   

What Theorem \ref{barmagnet2} and Corollary \ref{cor:lastcor} {\emph{do}} imply is that if there does exist a pair of distinct trees with the same chromatic symmetric function and a single centroid, then it must also share this special property with the trees in Example \ref{one_centroid}. That is, the partitions appearing as the $\theta_T$-images of 2-element sets must agree, though their relative location within the table might be ``scrambled.'' 

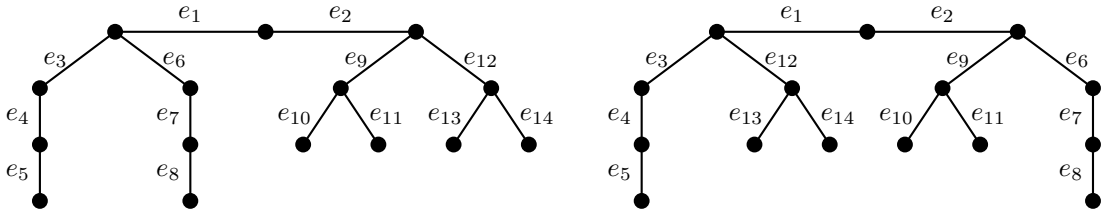
\begin{figure}[ht]
\centering
\begin{tikzpicture}[style=thick]

\draw (1,.75) coordinate (A13) node[left=3pt] {};
\draw (3,.75) coordinate (A14) node[left=3pt] {};
\draw (9,.75) coordinate (B13) node[left=3pt] {};
\draw (15,.75) coordinate (B14) node[left=3pt] {};

\draw (1,1.5) coordinate (A7) node[left=3pt] {};
\draw (3,1.5) coordinate (A8) node[left=3pt] {};
\draw (4.5,1.5) coordinate (A9) node[left=3pt] {};
\draw (5.5,1.5) coordinate (A10) node[left=3pt] {};
\draw (6.5,1.5) coordinate (A11) node[left=3pt] {};
\draw (7.5,1.5) coordinate (A12) node[left=3pt] {};
\draw (9,1.5) coordinate (B7) node[left=3pt] {};
\draw (10.5,1.5) coordinate (B8) node[left=3pt] {};
\draw (11.5,1.5) coordinate (B9) node[left=3pt] {};
\draw (12.5,1.5) coordinate (B10) node[left=3pt] {};
\draw (13.5,1.5) coordinate (B11) node[left=3pt] {};
\draw (15,1.5) coordinate (B12) node[left=3pt] {};

\draw (1,2.25) coordinate (A3) node[left=3pt] {};
\draw (3,2.25) coordinate (A4) node[left=3pt] {};
\draw (5,2.25) coordinate (A5) node[left=3pt] {};
\draw (7,2.25) coordinate (A6) node[left=3pt] {};
\draw (9,2.25) coordinate (B3) node[left=3pt] {};
\draw (11,2.25) coordinate (B4) node[left=3pt] {};
\draw (13,2.25) coordinate (B5) node[left=3pt] {};
\draw (15,2.25) coordinate (B6) node[left=3pt] {};

\draw (2,3) coordinate (A1) node[left=3pt] {};
\draw (6,3) coordinate (A2) node[left=3pt] {};
\draw (10,3) coordinate (B1) node[left=3pt] {};
\draw (14,3) coordinate (B2) node[left=3pt] {};

\draw (4, 3) coordinate (A0) node[left=3pt] {};
\draw (12, 3) coordinate (B0) node[left=3pt] {};

\draw(A1) -- (A0) node [midway, above] {$e_1$};
\draw(A0) -- (A2) node [midway, above] {$e_2$};
\draw(A1) -- (A4) node [midway, right] {$e_6$};
\draw(A1) -- (A3) node [midway, left]  {$e_3$};
\draw(A2) -- (A5) node [midway, left]  {$e_9$};
\draw(A2) -- (A6) node [midway, right] {$e_{12}$};
\draw(A3) -- (A7) node [midway, left]  {$e_4$};
\draw(A4) -- (A8) node [midway, left]  {$e_7$};
\draw(A5) -- (A9) node [midway, left]  {$e_{10}$};
\draw(A5) -- (A10) node [midway, right] {$e_{11}$};
\draw(A6) -- (A11) node [midway, left] {$e_{13}$};
\draw(A6) -- (A12) node [midway, right] {$e_{14}$};
\draw(A7) -- (A13) node [midway, left] {$e_5$};
\draw(A8) -- (A14) node [midway, left] {$e_8$};

\draw(B1) -- (B0) node [midway, above] {$e_1$};
\draw(B0) -- (B2) node [midway, above] {$e_2$};
\draw(B1) -- (B4) node [midway, right] {$e_{12}$};
\draw(B1) -- (B3) node [midway, left] {$e_3$};
\draw(B2) -- (B5) node [midway, left] {$e_9$};
\draw(B2) -- (B6) node [midway, right] {$e_6$};
\draw(B3) -- (B7) node [midway, left] {$e_4$};
\draw(B4) -- (B8) node [midway, left] {$e_{13}$};
\draw(B4) -- (B9) node [midway, right] {$e_{14}$};
\draw(B5) -- (B10) node [midway, left] {$e_{10}$};
\draw(B5) -- (B11) node [midway, right] {$e_{11}$};
\draw(B6) -- (B12) node [midway, left] {$e_7$};
\draw(B7) -- (B13) node [midway, left] {$e_5$};
\draw(B12) -- (B14) node [midway, left] {$e_8$};

  \fill (A0) circle (3pt)
  (A1) circle (3pt) 
  (A2) circle (3pt)
  (A3) circle (3pt)
  (A4) circle (3pt)
  (A5) circle (3pt)
  (A6) circle (3pt)
  (A7) circle (3pt)
  (A8) circle (3pt)
  (A9) circle (3pt)
  (A10) circle (3pt)
  (A11) circle (3pt)
  (A12) circle (3pt)
  (A13) circle (3pt)
  (A14) circle (3pt)
  (B0) circle (3pt)
  (B1) circle (3pt) 
  (B2) circle (3pt)
  (B3) circle (3pt)
  (B4) circle (3pt)
  (B5) circle (3pt)
  (B6) circle (3pt)
  (B7) circle (3pt)
  (B8) circle (3pt)
  (B9) circle (3pt)
  (B10) circle (3pt)
  (B11) circle (3pt)
  (B12) circle (3pt)
  (B13) circle (3pt)
  (B14) circle (3pt);
\end{tikzpicture}
\caption{Two trees with nearly identical $\theta_T$ images of 2-element sets}
\label{one_centroid}
\end{figure}
\tiny
\begin{center}
\begin{tabular}{@{}c@{}||@{\ }c@{\ }|@{\ }c@{\ }|@{\ }c@{\ }|@{\ }c@{\ }|@{\ }c@{\ }|@{\ }c@{\ }|@{\ }c@{\ }|@{\ }c@{\ }|@{\ }c@{\ }|@{\ }c@{\ }|@{\ }c@{\ }|@{\ }c@{\ }|@{\ }c@{\ }}
 & $e_1$ & $e_2$ & $e_3$ & $e_4$ & $e_5$ & $e_6$ & $e_7$ & $e_8$ & $e_9$ & $e_{10}$ & $e_{11}$ & $e_{12}$ & $e_{13}$\\ \hline \hline

$e_1$ & & & & & & & & & & & & & \\ \hline
$e_2$ & $(7, 7, 1)$ & & & & & & & & & & & & \\ \hline
$e_3$ & $(8, 4, 3)$ & $(7, 5, 3)$ & & & & & & & & & & & \\ \hline
$e_4$ & $(8, 5, 2)$ & $(7, 6, 2)$ & $(12, 2, 1)$ & & & & & & & & & & \\ \hline
$e_5$ & $(8, 6, 1)$ & $(7, 7, 1)$ & $(12, 2, 1)$ & $(13, 1, 1)$ & & & & & & & & & \\ \hline
$e_6$ & $(8, 4, 3)$ & $(7, 5, 3)$ & $(9, 3, 3)$ & $(10, 3, 2)$ & $(11, 3, 1)$ & & & & & & & & \\ \hline
$e_7$ & $(8, 5, 2)$ & $(7, 6, 2)$ & $(10, 3, 2)$ & $(11, 2, 2)$ & $(12, 2, 1)$ & $(12, 2, 1)$ & & & & & & & \\ \hline
$e_8$ & $(8, 6, 1)$ & $(7, 7, 1)$ & $(11, 3, 1)$ & $(12, 2, 1)$ & $(13, 1, 1)$ & $(12, 2, 1)$ & $(13, 1, 1)$ & & & & & & \\ \hline
$e_9$ & $(7, 5, 3)$ & $(8, 4, 3)$ & $(9, 3, 3)$ & $(10, 3, 2)$ & $(11, 3, 1)$ & $(9, 3, 3)$ & $(10, 3, 2)$ & $(11, 3, 1)$ & & & & & \\ \hline
$e_{10}$ & $(7, 7, 1)$ & $(8, 6, 1)$ & $(11, 3, 1)$ & $(12, 2, 1)$ & $(13, 1, 1)$ & $(11, 3, 1)$ & $(12, 2, 1)$ & $(13, 1, 1)$ & $(12, 2, 1)$ & & & & \\ \hline
$e_{11}$ & $(7, 7, 1)$ & $(8, 6, 1)$ & $(11, 3, 1)$ & $(12, 2, 1)$ & $(13, 1, 1)$ & $(11, 3, 1)$ & $(12, 2, 1)$ & $(13, 1, 1)$ & $(12, 2, 1)$ & $(13, 1, 1)$ & & & \\ \hline
$e_{12}$ & $(7, 5, 3)$ & $(8, 4, 3)$ & $(9, 3, 3)$ & $(10, 3, 2)$ & $(11, 3, 1)$ & $(9, 3, 3)$ & $(10, 3, 2)$ & $(11, 3, 1)$ & $(9, 3, 3)$ & $(11, 3, 1)$ & $(11, 3, 1)$ & & \\ \hline
$e_{13}$ & $(7, 7, 1)$ & $(8, 6, 1)$ & $(11, 3, 1)$ & $(12, 2, 1)$ & $(13, 1, 1)$ & $(11, 3, 1)$ & $(12, 2, 1)$ & $(13, 1, 1)$ & $(11, 3, 1)$ & $(13, 1, 1)$ & $(13, 1, 1)$ & $(12, 2, 1)$ & \\ \hline
$e_{14}$ & $(7, 7, 1)$ & $(8, 6, 1)$ & $(11, 3, 1)$ & $(12, 2, 1)$ & $(13, 1, 1)$ & $(11, 3, 1)$ & $(12, 2, 1)$ & $(13, 1, 1)$ & $(11, 3, 1)$ & $(13, 1, 1)$ & $(13, 1, 1)$ & $(12, 2, 1)$ & $(13, 1, 1)$
\label{Table 2}
\end{tabular}
\end{center}
\normalsize

\tiny
\begin{center}
\begin{tabular}{@{}c@{}||@{\ }c@{\ }|@{\ }c@{\ }|@{\ }c@{\ }|@{\ }c@{\ }|@{\ }c@{\ }|@{\ }c@{\ }|@{\ }c@{\ }|@{\ }c@{\ }|@{\ }c@{\ }|@{\ }c@{\ }|@{\ }c@{\ }|@{\ }c@{\ }|@{\ }c@{\ }}
 & $e_1$ & $e_2$ & $e_3$ & $e_4$ & $e_5$ & $e_6$ & $e_7$ & $e_8$ & $e_9$ & $e_{10}$ & $e_{11}$ & $e_{12}$ & $e_{13}$\\ \hline \hline

$e_1$ & & & & & & & & & & & & & \\ \hline
$e_2$ & $(7, 7, 1)$ & & & & & & & & & & & & \\ \hline
$e_3$ & $(8, 4, 3)$ & $(7, 5, 3)$ & & & & & & & & & & & \\ \hline
$e_4$ & $(8, 5, 2)$ & $(7, 6, 2)$ & $(12, 2, 1)$ & & & & & & & & & & \\ \hline
$e_5$ & $(8, 6, 1)$ & $(7, 7, 1)$ & $(12, 2, 1)$ & $(13, 1, 1)$ & & & & & & & & & \\ \hline
$e_6$ & $(7, 5, 3)$ & $(8, 4, 3)$ & $(9, 3, 3)$ & $(10, 3, 2)$ & $(11, 3, 1)$ & & & & & & & & \\ \hline
$e_7$ & $(7, 6, 2)$ & $(8, 5, 2)$ & $(10, 3, 2)$ & $(11, 2, 2)$ & $(12, 2, 1)$ & $(12, 2, 1)$ & & & & & & & \\ \hline
$e_8$ & $(7, 7, 1)$ & $(8, 6, 1)$ & $(11, 3, 1)$ & $(12, 2, 1)$ & $(13, 1, 1)$ & $(12, 2, 1)$ & $(13, 1, 1)$ & & & & & & \\ \hline
$e_9$ & $(7, 5, 3)$ & $(8, 4, 3)$ & $(9, 3, 3)$ & $(10, 3, 2)$ & $(11, 3, 1)$ & $(9, 3, 3)$ & $(10, 3, 2)$ & $(11, 3, 1)$ & & & & & \\ \hline
$e_{10}$ & $(7, 7, 1)$ & $(8, 6, 1)$ & $(11, 3, 1)$ & $(12, 2, 1)$ & $(13, 1, 1)$ & $(11, 3, 1)$ & $(12, 2, 1)$ & $(13, 1, 1)$ & $(12, 2, 1)$ & & & & \\ \hline
$e_{11}$ & $(7, 7, 1)$ & $(8, 6, 1)$ & $(11, 3, 1)$ & $(12, 2, 1)$ & $(13, 1, 1)$ & $(11, 3, 1)$ & $(12, 2, 1)$ & $(13, 1, 1)$ & $(12, 2, 1)$ & $(13, 1, 1)$ & & & \\ \hline
$e_{12}$ & $(8, 4, 3)$ & $(7, 5, 3)$ & $(9, 3, 3)$ & $(10, 3, 2)$ & $(11, 3, 1)$ & $(9, 3, 3)$ & $(10, 3, 2)$ & $(11, 3, 1)$ & $(9, 3, 3)$ & $(11, 3, 1)$ & $(11, 3, 1)$ & & \\ \hline
$e_{13}$ & $(8, 6, 1)$ & $(7, 7, 1)$ & $(11, 3, 1)$ & $(12, 2, 1)$ & $(13, 1, 1)$ & $(11, 3, 1)$ & $(12, 2, 1)$ & $(13, 1, 1)$ & $(11, 3, 1)$ & $(13, 1, 1)$ & $(13, 1, 1)$ & $(12, 2, 1)$ & \\ \hline
$e_{14}$ & $(8, 6, 1)$ & $(7, 7, 1)$ & $(11, 3, 1)$ & $(12, 2, 1)$ & $(13, 1, 1)$ & $(11, 3, 1)$ & $(12, 2, 1)$ & $(13, 1, 1)$ & $(11, 3, 1)$ & $(13, 1, 1)$ & $(13, 1, 1)$ & $(12, 2, 1)$ & $(13, 1, 1)$
\label{Table 3}
\end{tabular}
\end{center}
\normalsize
\begin{center}
Figure 5.8: $\theta_T$ images of 2-element sets for $T_1$ (above) and $T_2$ (below)
\end{center}

\end{document}